\RequirePackage{fix-cm}
\documentclass[smallextended]{svjour3}       
\smartqed  
\usepackage{graphicx}
\usepackage{amsmath,amssymb,amsfonts}
\usepackage[]{algorithm,algorithmic} 
\usepackage{mathtools}
\usepackage{tcolorbox}
\usepackage{bbm}
\usepackage{todonotes}
\usepackage{tensor}
\usepackage{caption}

\newcommand{\RR}{\mathbb{R}}
\newcommand{\EE}{\mathbb{E}}

\newcommand{\epi}{\mathrm{epi\ }}
\newcommand{\hypo}{\mathrm{hypo\ }}
\newcommand{\graph}{\mathrm{graph\ }}

\newcommand{\dom}{\mathrm{dom\ }}
\newcommand{\proj}{\mathrm{proj}}
\newcommand{\dist}{\mathrm{dist}}
\newcommand{\interior}{\mathrm{int\ }}
\newcommand{\closure}{\mathrm{cl\ }}

\DeclareMathOperator*{\argmin}{\arg\!\min}

\mathchardef\mhyphen="2D 
\DeclareMathOperator*{\argmax}{\arg\!\max}

\newcommand{\varSpace}{\mathcal{E}}
\newcommand{\extPos}{\overline{\mathbb{R}}_{++}}

\newcommand{\radTransSup}[1]{#1^{\Gamma}}
\newcommand{\biradTransSup}[1]{#1^{\Gamma\Gamma}}

\newcommand{\radTransSet}[1]{\Gamma#1}

\definecolor{blue}{gray}{0.0}

\allowdisplaybreaks

\begin{document}
	
	\title{Radial Duality\\ Part II: Applications and Algorithms\thanks{This material is based upon work supported by the National Science Foundation Graduate Research Fellowship under Grant No. DGE-1650441. This work was partially done while the author was visiting the Simons Institute for the Theory of Computing. It was partially supported by the DIMACS/Simons Collaboration on Bridging Continuous and Discrete Optimization through NSF grant \#CCF-1740425.}
	}
	
	
	\author{Benjamin Grimmer
	}
	
	
	\institute{B. Grimmer \at
		Johns Hopkins University, Baltimore, MD.
		\email{grimmer@jhu.edu}
	}
	
	\date{Received: date / Accepted: date}

	\maketitle
	
	\begin{abstract}
		The first part of this work established the foundations of a radial duality between nonnegative optimization problems, inspired by the work of Renegar~\cite{Renegar2016}. Here we utilize our radial duality theory to design and analyze projection-free optimization algorithms that operate by solving a radially dual problem. In particular, we consider radial subgradient, smoothing, and accelerated methods that are capable of solving a range of constrained convex and nonconvex optimization problems and that can scale-up more efficiently than their classic counterparts. These algorithms enjoy the same benefits as their predecessors, avoiding Lipschitz continuity assumptions and costly orthogonal projections, in our newfound, broader context. Our radial duality further allows us to understand the effects and benefits of smoothness and growth conditions on the radial dual and consequently on our radial algorithms.
	\end{abstract}
	
	\section{Introduction}\label{sec:intro}
The first part of this work~\cite{Grimmer2021-part1} established a theory of radial duality relating nonnegative optimization problems through a projective transformation, extending the ideas of Renegar~\cite{Renegar2016} from their origins in conic programming. We give a minimal overview here of our radial duality theory needed to begin algorithmically benefiting from it and then a fuller but terse summary in Section~\ref{sec:review} necessary to derive our radial optimization guarantees.

For a finite dimensional Euclidean space $\varSpace$, our three transformations of interest are the radial point transformation, radial set transformation, and upper radial function transformation, which are denoted by 
$$\Gamma(x,u) = (x,1)/u,$$
$$\Gamma S = \{\Gamma(x,u) \mid (x,u)\in S\},$$
$$\radTransSup{f}(y) = \sup\{v>0 \mid (y,v) \in \radTransSet{(\epi f)}\}$$
for any point $(x,u)\in\varSpace\times\RR_{++}$, set $S\subseteq\varSpace\times\RR_{++}$, and function $f \colon \varSpace\rightarrow \extPos$, respectively. Here $\extPos$ denotes the extended positive reals $\RR_{++}\cup\{0,+\infty\}$. It is immediate that the point and set transformations are dual since
$$ \Gamma\Gamma(x,u) = \Gamma \frac{(x,1)}{u} = \frac{(x/u,1)}{1/u} = (x,u).$$

Central to establishing our theory of radial duality is the characterization of exactly when this duality carries over to the function transformation. We say a function $f$ is {\it upper radial} if the perspective function $f^p(y,v) = v\cdot f(y/v)$ is upper semicontinuous and nondecreasing in $v\in\RR_{++}$. Moreover, it is {\it strictly upper radial} if it is strictly increasing in $v$ whenever $f^p(y,v)\in\RR_{++}$. The cornerstone theorem of our radial duality~\cite[Theorem 1]{Grimmer2021-part1} is that
\begin{equation} \label{eq:function-duality}
f = \biradTransSup{f} \iff f \text{ is upper radial.}
\end{equation}

The duality of the radial function transformation provides a duality between optimization problems {\color{blue} (see Section 4 of~\cite{Grimmer2021-part1})}.
For any strictly upper radial function $f\colon \varSpace\rightarrow \extPos$, consider the primal problem
\begin{equation}
p^* = \max_{x\in\varSpace} f(x). \label{eq:base-problem}
\end{equation}
Then the radially dual problem is given by
\begin{equation}
d^* = \min_{y\in\varSpace} \radTransSup{f}(y) \label{eq:radial-problem}
\end{equation}
and has $(\argmax f)\times\{p^*\}= \Gamma\left((\argmin \radTransSup{f})\times\{d^*\}\right)$. Thus maximizing $f$ is equivalent to minimizing $\radTransSup{f}$ and solutions can be converted between these problems by applying the radial point transformation $\Gamma$ or its inverse (which is also $\Gamma$ by duality). 

Importantly, the two nonnegative optimization problems~\eqref{eq:base-problem} and~\eqref{eq:radial-problem} can exhibit very different structural properties. For example, consider maximizing $f(x) = \sqrt{1 - \|x\|_2^2}_+$ which takes value zero outside the unit ball and has arbitrarily large gradients and Hessians as $x$ approaches the boundary of this ball. Its radial dual $\radTransSup{f}(y) = \sqrt{1+\|y\|_2^2}$ has a full domain with gradients and Hessians bounded in norm by one everywhere. Thus our radial duality theory poses an opportunity to extend the reach of many standard optimization algorithms reliant on such structure.
The previous works of Renegar~\cite{Renegar2016} and Grimmer~\cite{Grimmer2017-radial-subgradient} analyzing subgradient methods and Renegar~\cite{Renegar2019} employing accelerated smoothing techniques on a radial reformulation of the objective critically rely on the reformulation being uniformly Lipschitz continuous, which always occurs in the special cases of the radial dual that they consider.

\paragraph{Our Contributions.} 
This work leverages our radial duality theory to present and analyze projection-free radial optimization algorithms in this newfound, wider context than previous works were able to. 
Finding that a mild condition ensures the radial dual is uniformly Lipschitz continuous, we analyze a radial subgradient method for a broad range of non-Lipschitz primal problems with or without concavity. Observing that constraints radially transform into related gauges, we propose a radial smoothing method that takes advantage of this structure for concave maximization. Further, we find that our radial transformation extends smoothness on a level set of the primal to hold globally in the radial dual, which prompts our analysis of a radial accelerated method.
More important than these particular algorithms, this work aims to demonstrate the breadth of applications and algorithms that can be approached using our radial duality theory.

\paragraph{Outline.} We begin with a motivating example of the computational benefits and scalability that follow from designing algorithms based on the radial dual~\eqref{eq:radial-problem} in Section~\ref{sec:qp-example}. Then Section~\ref{sec:conditioning} formally establishes algorithmically useful properties of our radial dual, namely Lipschitz continuity, smoothness, and growth conditions. Finally, Section~\ref{sec:convex} addresses the convergence of our radial algorithms for concave maximization and Section~\ref{sec:nonconvex} addresses applications and guarantees in nonconcave maximization.

	\section{A Motivating Setting of Polyhedral Constraints}\label{sec:qp-example}
We begin by motivating the algorithmic usefulness of our radial duality by considering optimization with polyhedral constraints. Consider any maximization problem with upper semicontinuous objective $f\colon \RR^n \rightarrow \RR\cup\{-\infty\}$ and $m$ inequality constraints $a_i^Tx\leq b_i$ given by
\begin{equation} \label{eq:polyhedral}
\begin{cases}
\max_x & f(x)\\
\mathrm{s.t.}& Ax\leq b.
\end{cases}
\end{equation}

We assume this problem is feasible.
Then without loss of generality, we have {\color{blue} $0\in \interior (\{x \mid Ax\leq b,  f(x)>0\})$}. This can be achieved by computing any point $x_0$ in the relative interior of $ \{x \mid Ax\leq b,  f(x)\in \RR\}$ and then (i) translating the problem to place $x_0$ at the origin, (ii) adding a constant to the objective to ensure $f(0)>0$, and (iii) if needed, re-parameterizing the problem\footnote{Instead of using a re-parameterization, one can explicitly include equality constraints in our model. The details of this approach are given in Section~\ref{subsubsec:constraints}, where we see that equality constraints are unaffected by the radial dual.} to only consider the smallest subspace containing $\{x \mid Ax\leq b,  f(x)>0\}$. Note that doing this translation suffices to guarantee that any {\color{blue} upper semicontinuous,} concave $f$ will have $f_+(x):=\max\{f(x),0\}$ be strictly upper radial by~\cite[Proposition 11]{Grimmer2021-part1}.
We will only make the weaker assumption here that $f_+$ is strictly upper radial rather than the narrower case of it being concave. Then this problem can be reformulated as the following nonnegative optimization problem of our primal form~\eqref{eq:base-problem}
\begin{equation*}
\begin{cases}
\max_x & f_+(x)\\
\mathrm{s.t.}& Ax\leq b
\end{cases}
= \max_x \min_i\left\{f_+(x),\ \hat\iota_{a_i^Tx\leq b_i}(x)\right\}
\end{equation*}
where $\hat\iota_{a_i^Tx\leq b_i}(x)= \begin{cases}
+\infty & \text{if\ } a_i^Tx\leq b_i\\
0 & \text{if\ } a_i^Tx > b_i
\end{cases}$ is a nonstandard indicator function for each inequality constraint. Note that each $\hat\iota_{a_i^Tx\leq b_i}$ is strictly upper radial since $0$ is strictly feasible and so applying~\cite[Proposition 12]{Grimmer2021-part1} ensures the primal objective $ \min_i\left\{f_+(x),\ \hat\iota_{a_i^Tx\leq b_i}(x)\right\}$ is strictly upper radial. Then we can compute the radially dual optimization problem~\eqref{eq:radial-problem} using~\cite[Proposition 13]{Grimmer2021-part1} as
\begin{equation} \label{eq:polyhedral-dual}
\min_y \max_i\left\{f^\Gamma_+(y),\ a_i^Ty/b_i\right\}
\end{equation}
since the radial transformation of each nonstandard indicator function is linear
\begin{align*}
\hat\iota^{\Gamma}_{a_i^Tx\leq b_i}(y) &= \sup\left\{ v >0 \mid v\cdot \hat\iota_{a_i^Tx\leq b_i}(y/v)\leq 1\right\}\\
&= \sup\left\{ v >0 \mid a_i^T(y/v)> b_i\right\}\\
&= (a_i^Ty/b_i)_+.
\end{align*}
We drop the nonnegative thresholding on $a_i^Ty/b_i$ since $f^\Gamma_+(y)$ is nonnegative.

Importantly, the dual formulation~\eqref{eq:polyhedral-dual} is unconstrained, unlike the primal, since the primal inequality constraints have transformed into simple linear lower bounds on the radially dual objective. This dual further profits from the structure of its objective function as it is often globally Lipschitz continuous (a common property among radial duals that we will show in Proposition~\ref{prop:Lipschitz}) and has the simple form of a finite maximum. This radially dual structure gives us an algorithmic angle of attack not available in the primal problem.

\subsection{Quadratic Programming}
To make these benefits concrete, consider solving a generic quadratic program
\begin{equation} \label{eq:qp}
\begin{cases}
\max_x & 1-\frac{1}{2}x^T Qx - c^Tx\\
\mathrm{s.t.}& Ax\leq b
\end{cases}
\end{equation}
for some $Q\in\RR^{n\times n}$, $c\in\RR^n$, $A\in\RR^{m\times n}$, and $b\in \RR^m_{++}$. {\color{blue} Note this satisfies the needed condition $0\in \interior (\{x \mid Ax\leq b,  f(x)>0\})$ whenever $b>0$ as $f(0)=1$.}
We reformulate this problem as the following nonnegative optimization problem of the form~\eqref{eq:base-problem}
\begin{equation*} 
\begin{cases}
\max_x & 1-\frac{1}{2}x^T Qx - c^Tx\\
\mathrm{s.t.}& Ax\leq b
\end{cases}
= \max_x \min_i\left\{(1-\frac{1}{2}x^T Qx - c^Tx)_+,\ \hat\iota_{a_i^Tx\leq b_i}(x)\right\}.
\end{equation*}
Whenever this primal objective is strictly upper radial, the radial dual of our quadratic program is\footnote{Our calculation of the radial dual of the quadratic objective follows by definition as
	\begin{align*}
	(1-\frac{1}{2}x^T Qx - c^Tx)^\Gamma_+(y) &= \sup\left\{ v>0 \mid v\left(1-\frac{y^T Qy}{2v^2} - \frac{c^Ty}{v}\right)\leq 1 \right\}\\
	&= \sup\{ v>0 \mid v^2-\frac{1}{2}y^T Qy - (c^Ty+1)v\leq 0 \}\\
	&=\left(\frac{c^Ty+1 + \sqrt{(c^Ty+1)^2 +2y^TQy}}{2}\right)_+.
	\end{align*}
}
\begin{equation} \label{eq:qp-dual}
\min_y \max_i\left\{\left(\frac{c^Ty+1 + \sqrt{(c^Ty+1)^2 +2y^TQy}}{2}\right)_+,\  a_i^Ty/b_i\right\}
\end{equation}
where the first term in our maximum is set to zero if $(c^Ty+1)^2 +2y^TQy<0$ as can occur for nonconcave primal objectives. We find that our radial duality holds here whenever $\frac{1}{2}x^TQx > -1$ for all $Ax\leq b$. This captures two natural settings: (i) when the primal objective is concave (as $Q$ is positive semidefinite) or (ii) when the primal objective is nonconcave but has a compact feasible region (since we can rescale the objective to be $1-\lambda x^TQx/2 - \lambda c^Tx$ without changing the set of maximizers but ensuring $\frac{\lambda}{2}x^TQx > -1$ everywhere). Section~\ref{subsec:nc-obj} shows more generally that any differentiable objective with compact constraints can be rescaled to apply our radial duality theory.

We verify that our primal objective is strictly upper radial (and so our radial duality holds) for this upper semicontinuous objective by checking when $f^p(y,\cdot)$ is strictly increasing on its domain. The partial derivative with respect to $v$ of the perspective function
\begin{align*}
v\cdot \min_i\left\{(1-\frac{1}{2}(y/v)^T Q(y/v) - c^T(y/v))_+,\ \hat\iota_{a_i^Tx\leq b_i}(y/v)\right\}&\\
=\begin{cases}
v\left(1-\frac{y^T Qy}{2v^2} - \frac{c^Ty}{v}\right)_+ & \text{ if } A(y/v)\leq b\\
0 & \text{ otherwise}
\end{cases}
\end{align*}
is $1+\frac{y^TQy}{2v^2}$ at every feasible $y/v$. This is always positive (and hence the perspective function is increasing in $v$) exactly when every $x=y/v$ with $Ax\leq b$ has $\frac{1}{2}x^TQx>-1$.

\subsubsection{Quadratic Programming Numerics}
As previously noted, the radially dual formulation~\eqref{eq:qp-dual} is unconstrained and Lipschitz continuous despite the primal possessing neither of these properties. This differs from the structure found from taking a Lagrange dual~\cite{Dorn1960} or gauge dual~\cite{Freund1987}. As a result, our radial dual is well set up for the application of a subgradient method. We consider the following {\it radial subgradient method} with stepsizes $\alpha_k >0$ defined by Algorithm~\ref{alg:radial-subgradient-method}.
\begin{algorithm}[H] 
	\caption{The Radial Subgradient Method} 
	\label{alg:radial-subgradient-method} 
	\begin{algorithmic}[1] 
		\REQUIRE $f\colon \varSpace \rightarrow \extPos$, $x_0\in\dom f$, $T\geq 0$
		\STATE $(y_0, v_0) = \Gamma (x_0, f(x_0))$ \hfill {\it Transform into the radial dual}
		\FOR{$k=0\dots T-1$}
		\STATE $y_{k+1} = y_{k} - \alpha_k \zeta'_{k}, \text{\ where\ } \zeta'_{k}\in\partial_P \radTransSup{f}(y_{k})$ \hfill {\it Run the subgradient method}
		\ENDFOR
		\STATE $(x_T, u_T) = \Gamma (y_T, \radTransSup{f}(y_T))$  \hfill {\it Transform back to the primal}
	\end{algorithmic}
\end{algorithm}
Further noting that the radially dual problem is a finite maximum of simple smooth Lipschitz functions, we can apply the smoothing ideas of Nesterov~\cite{Nesterov2005}. Perhaps the clearest description of these techniques is given by Beck and Teboulle~\cite{Beck2012}. In particular, for any fixed $\eta>0$, we consider the smooth function given by taking a ``soft-max''
\begin{equation} \label{eq:qp-smoothing}
g_\eta(y) = \eta \log\left(\exp\left(\frac{c^Ty+1 + \sqrt{(c^Ty+1)^2 +2y^TQy}}{2\eta}\right)+\sum_{i=1}^m\exp\left(\frac{a_i^Ty}{b_i\eta}\right)\right)
\end{equation}
which approaches our radially dual objective as $\eta\rightarrow 0$. Then we can minimize the radial dual up to accuracy $O(\eta)$ by minimizing this smoothed objective. Doing so with Nesterov's accelerated method gives the following {\it radial smoothing method} defined by Algorithm~\ref{alg:radial-smoothing-method} (a similar radial algorithm was employed by Renegar~\cite{Renegar2019} showing that the transformation of any hyperbolic programming problem also admits a smoothing that can be efficiently minimized).
\begin{algorithm}[H] 
	\caption{The Radial Smoothing Method} 
	\label{alg:radial-smoothing-method} 
	\begin{algorithmic}[1] 
		\REQUIRE $f\colon \varSpace \rightarrow \extPos$, $x_0\in\dom f$, $\eta>0$, $L_\eta>0$, $T\geq 0$
		\STATE $(y_0, v_0) = \Gamma (x_0, f(x_0))$ and $\tilde y_{0}=y_0$ \hfill {\it Transform into the radial dual}
		\STATE Let $g_\eta(y)$ denote an $\eta$-smoothing of $f^\Gamma(y)$
		\FOR{$k=0\dots T-1$}
		\STATE $\tilde y_{k+1} = y_k - \nabla g_\eta(y_k)/L_\eta$ \hfill {\it Run the accelerated method}
		\STATE $y_{k+1} = \tilde y_{k+1} + \frac{k-1}{k+2}(\tilde y_{k+1}-\tilde y_k)$ 
		\ENDFOR
		\STATE $(x_T, u_T) = \Gamma (y_T, \radTransSup{f}(y_T))$  \hfill {\it Transform back to the primal}
	\end{algorithmic}
\end{algorithm}

The per iteration cost of these radial methods is controlled by the cost of evaluating one subgradient of the radially dual objective~\eqref{eq:qp-dual} or one gradient of our smoothing of the radially dual objective~\eqref{eq:qp-smoothing}. Both of these can be done efficiently in closed form in terms of matrix-vector products with $A$ and $Q$. Despite this low iteration cost, a feasible primal solution $(x_k,u_k)=\Gamma (y_k, f^\Gamma(y_k))$ is known at every iteration. Convergence guarantees for the radial subgradient and smoothing methods for concave maximization are given later in Sections~\ref{subsec:subgradient} and~\ref{subsec:smoothing}.

Classic optimization algorithms that preserve feasibility at every iteration tend to have much higher iteration costs. Here we compare with three of the most standard first-order methods that enforce feasibility: projected gradient descent (or rather, projected gradient ascent)
$$ x_{k+1} = \proj_{\{x\mid Ax\leq b\}}\left(x + \nabla f(x)/L\right),$$
an accelerated projected gradient method
\begin{equation*}
\begin{cases}
\tilde x_{k+1} &= \proj_{\{x\mid Ax\leq b\}}\left(x_k + \nabla f(x_k)/L\right)\\
x_{k+1} &= \tilde x_{k+1} + \frac{k-1}{k+2}(\tilde x_{k+1}-\tilde x_k),
\end{cases}
\end{equation*}
and the Frank-Wolfe method\footnote{Quadratic programming was the original motivating setting for Frank-Wolfe~\cite{Frank1956}.} with stepsize sequence $\beta_k>0$
\begin{equation*}
\begin{cases}
\tilde x_{k+1} &\in \argmax_x\left\{\nabla f(x_k)^T x \mid Ax\leq b\right\}\\
x_{k+1} &= x_{k} + \beta_k(\tilde x_{k+1}-x_k).
\end{cases}
\end{equation*}

All three of these methods require solving a subproblem at each iteration. The projected gradient and accelerated gradient methods require repeated projection onto the polyhedron $\{x\mid Ax\leq b\}$, which is itself an instance of~\eqref{eq:qp} specialized to $Q=I$. The Frank-Wolfe method requires repeatedly solving a linear program over this polyhedron. Both of these operations are far more expensive than the matrix-vector products required by the radial subgradient and smoothing methods but may allow them to have a greater improvement in objective value per iteration.

To weigh this tradeoff, we consider running these five algorithms on synthetic quadratic programs given by drawing two matrices $A\in \RR^{m\times n}$ and $P\in\RR^{n\times 100}$ and a vector $c\in\RR^n$ with i.i.d.~Guassian entries and setting $Q=PP^T$ and all $b_i=1$. Then we run each algorithm for $30$ minutes on instances of size $(n,m)\in\{(400,1600),(800,3200),(1600,6400)\}$. Our numerical experiments are conducted on a four-core Intel i7-6700 CPU using Julia 1.4.1 and Gurobi {\color{blue} 9.1.1} to solve any subproblems\footnote{The source code is available at github.com/bgrimmer/Radial-Duality-QP-Example}.
For each method, we set $x_0=0$ and use the following choice of stepsizes: the projected and accelerated gradient methods use $L=\lambda_{max}(Q)$, the Frank-Wolfe method uses an exact linesearch $\beta_k=\min\left(\frac{\nabla f(x_k)^T(\tilde x_{k+1}-x_k)}{\|P^T(\tilde x_{k+1}-x_k)\|^2},1\right) $, the radial subgradient method uses the Polyak stepsize $\alpha_k=\frac{f^\Gamma(y_k)-d^*}{\|\zeta_k'\|^2}$, and the radial smoothing method fixes $L_\eta=0.1\max\{\|a_i/b_i\|^2\}/\eta$ and $\eta\in\{10^{-8},5\times 10^{-7},10^{-7}\}$ for each of our three problem sizes.

The best primal objective value seen by each method is shown in realtime in Figure~\ref{fig:qp-comparison}.
First, we remark on the total number of iterations completed by each method in the allotted half hour, shown in the following table.
{\color{blue} \begin{center}
		\begin{tabular}{|c || c c c|} 
			\hline
			$(n,m)$ & $(400,1600)$ & $(800,3200)$ & $(1600,6400)$ \\ [0.5ex] 
			\hline \hline
			Projected Gradient & 16,520 iter. & 3,900 iter. & 152 iter. \\ 
			\hline
			Accelerated Gradient & 16,490 iter. & 3,425 iter. & 70 iter. \\ 
			\hline
			Frank-Wolfe & 1,033 iter. & 216 iter. & 35 iter. \\
			\hline
			Radial Subgradient & 5,635,174 iter. & 907,797 iter. & 225,971 iter. \\
			\hline
			Radial Smoothing & 3,344,776 iter. & 448,785 iter. & 111,871 iter. \\
			\hline
		\end{tabular}
\end{center}}
In our largest problem setting $(n,m)=(1600,6400)$, which has approximately ten million nonzeros, the projected gradient, accelerated gradient, and Frank-Wolfe methods complete 35-152 steps within our time budget whereas our radial methods take hundreds of thousands of steps. For our smallest instance $(n,m)=(400,1600)$, the accelerated gradient method quickly reaches high accuracy. However, for our moderate-sized instance $(n,m)=(800,3200)$, the classic methods begin to fall off with the radial smoothing method and accelerated method performing comparably up to accuracy $O(\eta)$. For our largest instance $(n,m)=(1600,6400)$, the methods relying on orthogonal projection and linear optimization have their progress substantially slowed due to their high iteration cost. Our radial algorithms appear to provide a more scalable approach.
\begin{figure}
	\includegraphics[width=\linewidth]{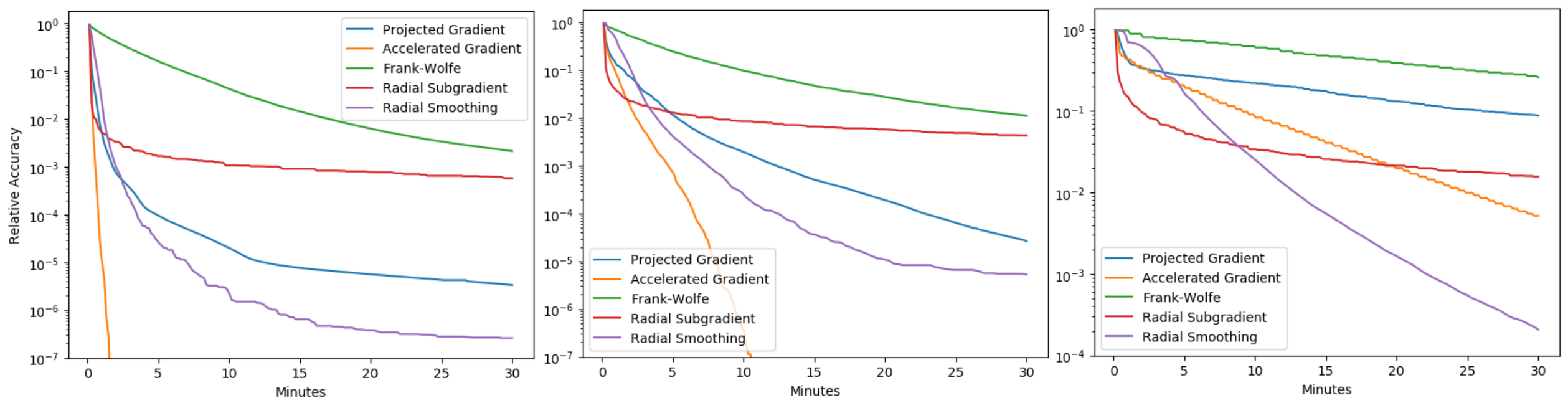}
	\caption{The minimum relative accuracy $\frac{p^*-f(x_k)}{p^*}$ of~\eqref{eq:qp}, with sizes $(n,m)$ equal to $(400,1600),(800,3200),(1600,6400)$ from left to right, seen by the projected gradient, accelerated gradient, Frank-Wolfe, radial subgradient and radial smoothing methods over $30$ minutes.} \label{fig:qp-comparison}
\end{figure}
Throughout our experiments, the radial smoothing method outperforms the radial subgradient method by a couple of orders of magnitude. This agrees with our convergence theory showing that the radial subgradient method converges at a $O(1/\epsilon^2)$ rate while the smoothing technique enables $O(1/\epsilon)$ convergence, presented in Sections~\ref{subsec:subgradient} and~\ref{subsec:smoothing}, respectively. 

{\color{blue} Further comparisons can be made with customized, scalable QP solvers like OSQP~\cite{osqp}, which is based on ADMM and provides approximate primal and dual solutions. 
	In Appendix~\ref{appendix}, we outline how dual solutions can be extracted from the radial smoothing method for the sake of comparison with OSQP. The quality of some primal $x$ and dual $v$ as an approximate KKT solution can be measured in terms of their primal feasibility $ \epsilon_{prim} = \max\{a_i^Tx - b_i, 0\}$, dual feasibility $ \epsilon_{dual} = \|Qx + c + A^Tv\|_\infty $, and complementary slackness $ \epsilon_{comp} = \|(Ax-b)\cdot v\|_\infty $.
	OSQP always has $\epsilon_{comp}=0$ but guarantees neither primal nor dual feasibility at its iterates. Our radial smoothing method always has a feasible primal solution $\epsilon_{prim}=0$, but guarantees neither dual feasibility nor complementary slackness. As a result, only limited conclusions can be drawn between these two methods. Figure~\ref{fig:osqp-comparison} shows the convergence of each OSQP and radial smoothing (with $\eta=10^{-4}$). Although asymptotically OSQP appears to converge faster, throughout the experiment's runtime it is debatable which method's certificates are preferable. Further numerical testing and the design of radial methods/theory focused on KKT attainment are deferred to future works. 
	Note at each iteration, OSQP solves a linear system rather than just relying on matrix multiplications. Numerically, this results in OSQP completing $10,503$ steps whereas our radial smoothing method completed $113,532$.}
\begin{figure}
	\includegraphics[width=\linewidth]{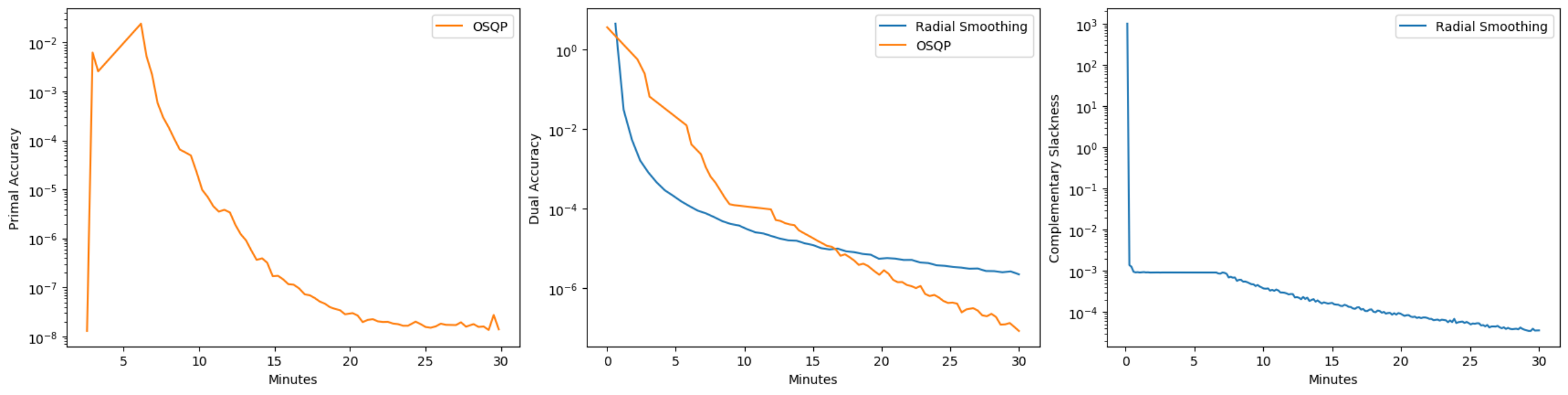}
	\caption{{\color{blue}Convergence of $\epsilon_{prim},\epsilon_{dual},\epsilon_{comp}$ on a random QP of size $(n,m)=(1600,6400)$.}} \label{fig:osqp-comparison}
\end{figure}

\subsection{Broader Computational Advantages from the Radially Dual Problem}
We conclude this motivating section with a high-level discussion of the computational advantages we see in optimizing over the radially dual problem. 

\subsubsection{Maintaining Primal Feasible Iterates Without Costly Projections} \label{subsubsec:constraints}
Here we generalize the setting of polyhedral constraints considered by~\eqref{eq:polyhedral}. After a translation, any convex constraints can be expressed as the intersection of a convex set $S\subseteq\varSpace$ with $0\in \interior S$ and a subspace $T=\{x\in \varSpace \mid Ax=0\}$. Consider any primal problem with strictly upper radial objective $f$ given by
$$ \begin{cases}\max & f(x)\\ \text{s.t.} & x\in S\\ & Ax=0 \end{cases} = \max_{x\in\varSpace} \min\{f(x),\hat\iota_{S}(x),\hat\iota_{T}(x)\}$$
where $\hat\iota_{S}(x)= \begin{cases}
+\infty & \text{if\ } x\in S\\
0 & \text{if\ } x\not\in S.
\end{cases}$ Then the radially dual problem is
$$ \min_{y\in\varSpace} \max\{\radTransSup{f}(y), \gamma_S(y), \gamma_{T}(y)\} =  \begin{cases}\min & \max\{f^\Gamma(y), \gamma_S(y)\}\\ \text{s.t.} & Ay=0 \end{cases}$$
where $\gamma_S(y) = \inf\{\lambda\geq 0 \mid y\in\lambda S\}$ denotes the Minkowski gauge since
\begin{align*}
\hat\iota_S^\Gamma(y) = \sup\{v>0 \mid v\cdot \hat\iota_S(y/v)\leq 1\} &= \sup\{v>0 \mid y/v\not\in S\}\\&= \inf\{\lambda>0 \mid y\in \lambda S\}=\gamma_S(y).
\end{align*}
{\color{blue} The last line above uses that $S$ is convex and contains $0$.}
Having multiple set constraints $S_1\dots S_n$ in the primal $\max_{x\in S_1\cap\dots\cap S_n} f(x)$ simply adds more terms to the dual's maximum of $\min_{y\in\varSpace} \max\{f^\Gamma(y), \gamma_{S_i}(y)\}$.

This formulation allows algorithms to maintain a feasible primal solution at each iteration without requiring costly subproblems relating to $S$. Instead, a primal feasible solution can be recovered from any radial dual solution $y\in \varSpace$ with $Ay=0$ as $x=y/\max\{\radTransSup{f}(y), \gamma_S(y)\}\in S\cap T$ {\color{blue} since $0\in S\cap T$}. Algorithmically, this replaces the need for orthogonal projections onto the feasible region $S\cap T$ with the cheaper operations of orthogonally projecting onto the subspace $T$ and evaluating the gauge of $S$. This computational gain was one of the key contributions identified by~\cite{Renegar2016} and was central to the motivation of~\cite{Renegar2019,Grimmer2017-radial-subgradient} as well as being a motivation of this work.



\subsubsection{Handling Nonconcave Objectives and Nonconvex Constraints}
Our calculation of the radial dual for quadratic programming did not fundamentally rely on concavity as it also applies to nonconcave problems with a bounded feasible region. Indeed one of the key insights from the first part of this work was divorcing the idea of radial transformations from relying on notions of convexity or concavity. In Section~\ref{subsec:nc-obj}, we discuss several nonconcave primal maximization problems where radial duality holds, generalizing the above reasoning to star-convex constraints and covering important areas like nonconvex regularization and optimization with outliers. 

\subsubsection{Efficiently Evaluating Generic Radial Duals}
{\color{blue}
	In many structured settings, we can exactly evaluate the radial transformation with cost comparable to a single evaluation of $f$. Table~\ref{tab:constraints} gives formulas for the gauge of any norm, halfspace, polynomial, or semidefinite programming constraint. Note that the gauge of an intersection of several of these constraints is simply the maximum of each constraint's gauge formula. Similarly, Table~\ref{tab:functions} gives formulas for the radial dual of many common function classes.
	\begin{table}
		{\color{blue}\begin{tabular}{|l|c|}
				\hline Set $S$ (assumed star convex with $0\in \mathrm{int}\ S$) & $\gamma_S(y) = \hat{\iota}^\Gamma_S(y)$ \\
				\hline\hline Norm Constraints $\{x \mid \|x\|\leq b\}$ & $\|y\|/b$\\
				Halfspace Constraints $\{x \mid a^Tx\leq b\}$ & $(a^Ty/b)_+$ \\
				Quadratic Constraints $\{x \mid \frac{1}{2}x^TQx + p^Tx \leq b\}$ & $\left(\frac{p^Ty + \sqrt{(p^Ty)^2 + 2by^TQy}}{2c}\right)_+$\\
				Polynomial Constraints $\{x \mid p(x)\leq 0\}$ & Polynomial Root Finding\\
				Semidefinite Constraints $\{x \mid \mathcal{A}x - B \preceq 0\}$ & $\lambda_{max}(B^{-1}\mathcal{A}y)$ \\ \hline
			\end{tabular}
			\caption{Common families of constraints with closed-form descriptions of their gauge.}\label{tab:constraints}}
	\end{table}
	\begin{table}
		{\color{blue}\begin{tabular}{|l|c|}
				\hline Function $f$ (assumed upper radial, $f(0)>0$) & $f^\Gamma(y)$ \\
				\hline\hline Norms $\|x\|$ & $\hat{\iota}_{\|x\|\leq 1}(y)$\\
				Linear Functions $(a^Tx+b)_+$ & $((1-a^Ty)/b)_+$ \\
				Quadratic Functions $\left(\frac{1}{2}x^TQx + p^Tx + b\right)_+$ & $\left(\frac{1-p^Ty + \sqrt{(1-p^Ty)^2 - 2by^TQy}}{2c}\right)_+$\\
				Polynomial Functions $p(x)_+$ & Polynomial Root Finding\\ \hline
			\end{tabular}
			\caption{Common function classes with closed-form descriptions of their radial dual.}\label{tab:functions}}
	\end{table}
}

Outside these common families, $f^\Gamma$ does not have a closed-form formula. However, numerically evaluating $f^\Gamma(y)$ can be done by bisection whenever $f$ is upper radial: {\color{blue} Note zero is a trivial lower bound on the value of $f^\Gamma(y)$ and an upper bound can be computed via exponential back-off, finding the first integer $i$ with $2^i\cdot  f(y/2^i)>1$ since all $2^i>f^\Gamma(y)$ have this property. Having nondecreasing $v\cdot f(y,v)$ ensures bisection will linearly converge to the unique intermediate value of $v$ where $vf(y/v)$ passes above one. This exponential back-off and subsequent bisection will each require a logarithmic number of function evaluations (both will reach standard machine precision within $\approx 30$ steps). Consequently, even when closed-forms are not available, the radial dual is at most only moderately more expensive to calculate than the primal.} Even if $f$ is not upper radial, $f^\Gamma(y)$ may still be tractable to compute. For example, any polynomial $f$ has evaluation of $f^\Gamma$ amount to computing a polynomial's largest root.
Once $f^\Gamma(y)$ has been computed, its gradient and Hessian can be readily computed from~\eqref{eq:gradient-formula} and~\eqref{eq:hessian-formula}.

\subsubsection{Improving Conditioning and Problem Structure}

As a final motivating example of the structural advantages of taking the radial dual, consider the following Poisson inverse problem.
Given linear measurements with Poisson distribution noise $b_i \sim \mathtt{Poisson}(a_i^Tx)$, the maximum likelihood estimator is given by maximizing
$$\mathcal{L}(x) := \begin{cases} 
\sum_i b_i\log(a_i^Tx) - a_i^Tx & \text{ if all } a_i^Tx>0\\
-\infty & \text{ otherwise.} \end{cases}
$$ 
Then given any convex regularizer $r(x)$ and constraint set $S\subseteq \RR^n$, we formulate a Poisson inverse problem as
\begin{equation} \label{eq:poisson-objective}
\max_{x\in S} \mathcal{L}(x) -r(x).
\end{equation}
This type of problem arises in image processing (see~\cite{Bertero2009} for a survey of applications from astronomy to medical imaging) as well as in network diffusion and time series modeling (see the many references in~\cite{He2016}). Although this problem is concave, the blow-up from the logarithmic terms prevents standard first-order methods from being applied. Provided the regularization $r$ and constraints $S$ are sufficiently simple, customized primal-dual~\cite{He2016} or Bregman methods~\cite{Bauschke2017,Mukkamala2020} provide a powerful tactic for solving this problem.

For generic $S$ and $r$, our radial duality can be applied. Given any {\color{blue} $x_0\in\interior (\dom \mathcal L \cap S)$} and $u_0 < \mathcal{L}(x_0)-r(x_0)$, we can reformulate this objective function to be strictly upper radial via a simple translation and truncation. We consider the equivalent problem of 
$$ \max_{x\in \RR^n} \min\{(\mathcal{L}(x+x_0) -r(x+x_0)- u_0)_+, \hat\iota_{S}(x+x_0)\}. $$
Then we can employ our radial duality machinery using~\cite[Proposition 11]{Grimmer2021-part1} since our translated and truncated objective is concave with $0$ strictly in its domain. The radial dual here is defined everywhere $\dom f^\Gamma = \RR^n$ and globally uniformly Lipschitz continuous (see Proposition~\ref{prop:Lipschitz}). {\color{blue} Moreover if $S=\RR^n$, $r(x)$ is twice continuously differentiable and $\mathcal{L}-r$ has bounded level sets\footnote{{\color{blue} For example, if either $\{a_i\}$ spans $\RR^n$ or the regularizer $r(x)$ has bounded level sets.}}, the radial dual has globally Lipschitz continuous gradient (see Corollary~\ref{cor:smoothness}).} The primal formulation is none of these.
Note that different translations of the objective (here corresponding to a different choice of $(x_0,u_0)$) produce different radial duals, which in turn can have very different global Lipschitz and smoothness constants.
%

\subsection{Notation and Review} \label{sec:review}
We consider functions $f \colon \varSpace \rightarrow \extPos$, where $\extPos = \RR_{++}\cup\{0,+\infty\}$ denotes the ``extended positive reals''. Here $0$ and $+\infty$ are the limit objects of $\RR_{++}$, mirroring the roles of $-\infty$ and $+\infty$ in the extended reals.
The effective domain, graph, epigraph, and hypograph of such a function are
\begin{align*}
\dom f &:= \{x\in\varSpace \mid f(x)\in\RR_{++}\},\\
\graph f &:= \{(x,u)\in \varSpace \times \RR_{++} \mid f(x)= u\},\\
\epi f &:= \{(x,u)\in \varSpace \times \RR_{++} \mid f(x)\leq u\},\\
\hypo f &:= \{(x,u)\in \varSpace \times \RR_{++} \mid f(x)\geq u\}.
\end{align*}
We say a function $f \colon \varSpace \rightarrow \extPos$ is upper (lower) semicontinuous if $\hypo f$ ($\epi f$) is closed with respect to $\varSpace \times \RR_{++}$.
Equivalently, a function is upper semicontinuous if for all $x\in\varSpace$, $f(x)=\limsup_{x'\rightarrow x} f(x')$ and lower semicontinuous if $f(x)=\liminf_{x'\rightarrow x} f(x')$.
We say a function $f \colon \varSpace \rightarrow \extPos$ is concave (convex) if $\hypo f$ ($\epi f$) is convex.
The set of {\it convex normal vectors} of a set $S\subseteq \varSpace\times \RR$ at some $(x,u)\in S$ is denoted by
$$ N^C_{S}((x,u)) := \{(\zeta,\delta) \mid (\zeta,\delta)^T((x,u) - (x',u')) \geq 0\ \forall (x',u')\in S\}.$$
Then the {\it convex subdifferential} and {\it convex supdifferential} of a function $f$ are
\begin{align*}
\partial_C f(x) &:= \{\zeta \mid (\zeta,-1)\in N^C_{\epi f}((x,f(x)))\},\\
\partial^C f(x) &:= \{\zeta \mid (-\zeta,1)\in N^C_{\hypo f}((x,f(x)))\}.
\end{align*}
The proximal normal vectors and differentials of a set $S$ or function $f$ are
\begin{align*}
N^P_{S}((x,u)) &:= \{(\zeta,\delta) \mid (x,u)\in \proj_S( (x,u) + \epsilon(\zeta,\delta)) \text{\ for some\ }\epsilon>0\},\\
\partial_P f(x) &:= \{\zeta \mid (\zeta,-1)\in N^P_{\epi f}((x,f(x)))\},\\
\partial^P f(x) &:= \{\zeta \mid (-\zeta,1)\in N^P_{\hypo f}((x,f(x)))\}.
\end{align*}

\noindent {\it Dual Families of Functions.}
Most of our theory characterizing the radial transformation relies on the given function being (strictly) upper radial. Recall that \cite[Proposition 6]{Grimmer2021-part1} shows an upper semicontinuous function $f$ is upper radial (that is, our radial duality $f^{\Gamma\Gamma}=f$ holds) if and only if all $(x,u)\in\hypo f$ and $(\zeta,\delta)\in N^P_{\hypo f}((x,u))$ satisfy
\begin{equation} \label{eq:radial-NS-condition}
(\zeta,\delta)^T(x,u)\geq 0.
\end{equation}
Geometrically, this corresponds to the origin lying below all of the hyperplanes induced by proximal normal vectors of the hypograph.
Similarly, \cite[Proposition 8]{Grimmer2021-part1} ensures a continuously differentiable function $f$ is strictly upper radial if all $x\in\dom f$ satisfy
\begin{equation} \label{eq:differentiable-strict-radial-NS-condition}
(\nabla f(x),-1)^T(x,u)< 0.
\end{equation}
For concave functions, being upper radial corresponds to the origin lying in the function's domain. In particular, \cite[Proposition 11]{Grimmer2021-part1} ensures an upper semicontinuous concave function $f$ is strictly upper radial if
\begin{equation} \label{eq:concave-strict-radial-NS-condition}
0\in\interior\{x\mid f(x)>0\}.
\end{equation}

Assuming strict upper radiality holds, the following are radially dual
\begin{align}
f\text{ is upper semicontinuous } & \iff \radTransSup{f}\text{ is lower semicontinuous,}\label{eq:semicontinuity-duality}\\
f\text{ is continuous } & \iff \radTransSup{f}\text{ is continuous,}\label{eq:continuity-duality}\\
f\text{ is concave } & \iff \radTransSup{f}\text{ is convex,}\label{eq:convexity-duality}
\end{align}
where these follow from~\cite[Propositions 15, 17]{Grimmer2021-part1}. 
{\color{blue} We say a functions mapping into the extended positives is $k$ times differentiable if it is $k$ times differentiable at each $x$ in $\dom f$.} For differentiable functions satisfying~\eqref{eq:differentiable-strict-radial-NS-condition}, \cite[Proposition 21]{Grimmer2021-part1} shows
\begin{align}
f\text{ is } k \text{ times differentiable} & \iff \radTransSup{f}\text{ is } k \text{ times differentiable,} \label{eq:differentiable-duality}\\
f\text{ is analytic} & \iff \radTransSup{f}\text{ is analytic.} \label{eq:analytic-duality}
\end{align}

\noindent {\it Relating Extreme Points, (Sub)Gradients, and Hessians.}
We recall a few bijections relating functions and their radial transformations. For any strictly upper radial $f$, \cite[Lemma 2]{Grimmer2021-part1} ensures
\begin{equation}\label{eq:upper-graph-bijection}
\epi\radTransSup{f} = \Gamma(\hypo f).
\end{equation} 
Further, \cite[Lemma 3]{Grimmer2021-part1} shows for any continuous strictly upper radial function, the following pair of bijections between graphs and domains hold
\begin{equation}\label{eq:graph-bijection}
\graph\radTransSup{f} = \Gamma(\graph f),
\end{equation} 
\begin{equation}\label{eq:domain-bijection}
y\in\dom\radTransSup{f} \iff y/\radTransSup{f}(y)\in\dom f.
\end{equation} 
Then \cite[Propositions 24, 25]{Grimmer2021-part1} shows that the radial point transformation relates the maximizers of a strictly upper radial function $f$ to the minimizers of $f^\Gamma$ as well as relates their stationary points
\begin{align}
\argmin \radTransSup{f}\times\{\inf \radTransSup{f}\}& = \radTransSet{\left(\argmax f \times\{\sup f\}\right)}, \label{eq:minimizer-conversion}\\
\{(y,\radTransSup{f}(y)) \mid 0\in\partial_P\radTransSup{f}(y)\} &= \radTransSet{\{(x,f(x)) \mid 0\in\partial^Pf(x)\}}. \label{eq:stationary-conversion}
\end{align}
In particular, for any upper semicontinuous, strictly upper radial $f$, the convex and proximal subgradients of its upper radial transformation are given by
\begin{align}
\partial_C\radTransSup{f}(y) &=  \left\{\frac{\zeta}{(\zeta,\delta)^T(x,u)} \mid \begin{bmatrix} \zeta \\ \delta\end{bmatrix}\in N^C_{\hypo f}((x,u)),\ (\zeta,\delta)^T(x,u)>0\right\}\label{eq:convex-subgradient-formula}\\ 
\partial_P\radTransSup{f}(y) &=\left\{\frac{\zeta}{(\zeta,\delta)^T(x,u)} \mid \begin{bmatrix} \zeta \\ \delta\end{bmatrix}\in N^P_{\hypo f}((x,u)),\ (\zeta,\delta)^T(x,u)>0\right\} \label{eq:proximal-subgradient-formula}
\end{align}
where $(x,u) = \Gamma(y,\radTransSup{f}(y))$ by \cite[Propositions 19, 20]{Grimmer2021-part1}.
Further, if $f$ is continuously differentiable and satisfies~\eqref{eq:differentiable-strict-radial-NS-condition}, \cite[Proposition 21]{Grimmer2021-part1} shows the gradient of the upper radial transformation at $y=x/f(x)$ is
\begin{equation}\label{eq:gradient-formula}
\nabla \radTransSup{f}(y) = \frac{\nabla f(x)}{(\nabla f(x),-1)^T(x,f(x))}.
\end{equation}
If in addition we suppose $f$ is twice continuously differentiable around $x$, \cite[Proposition 22]{Grimmer2021-part1} shows the Hessian of the upper radial transformation is
\begin{equation}\label{eq:hessian-formula}
\nabla^2 \radTransSup{f}(y) = \frac{f(x)}{(\nabla f(x),-1)^T(x,f(x))}\cdot J\nabla^2 f(x)J^T
\end{equation}
where $J = I -\frac{\nabla f(x)x^T}{(\nabla f(x),-1)^T(x,f(x))}$.

	\section{Conditioning of the Radially Dual Problem}\label{sec:conditioning}
As we have seen, the radial dual often enjoys structural properties missing from the primal.
In the following three subsections, we characterize the Lipschitz continuity, smoothness, and growth conditions of the radially dual problem. 

\subsection{Lipschitz Continuity of the Radially Dual Problem}
We say a function $f$ is uniformly $M$-Lipschitz continuous if for all $x,x'\in\varSpace$,
$$ |f(x) - f(x')| \leq M\|x-x'\|.$$
For any lower semicontinuous function $f\colon\varSpace\rightarrow \RR_{++}\cup\{\infty\}$, $M$-Lipschitz continuity is equivalent to all proximal subgradients $\zeta\in\partial_P f(x)$ having norm bounded by $M$~\cite[Theorem 1.7.3]{Clarke1998-nonsmoothanalysis}.

Lipschitz continuity plays an important role in the analysis of many first-order methods for nonsmooth optimization. The previous works~\cite{Renegar2016,Renegar2019,Grimmer2017-radial-subgradient} critically rely on their radially reformulated objective being uniformly Lipschitz. Here we present a general characterization of when the radial transformation of a function is uniformly Lipschitz.
To take advantage of the second characterization of Lipschitz continuity above, we need to ensure $\radTransSup{f}$ maps into $\RR_{++}\cup\{\infty\}$. The following simple assumption is equivalent to this (by the definition of the upper radial transformation): for all $y\in\varSpace$,
$\lim_{v\rightarrow 0} v\cdot f(y/v) = 0.$

This condition is always the case when $f$ is bounded above as will typically be the case for our primal maximization problem. Under this condition, we find that the Lipschitz continuity of $\radTransSup{f}$ is controlled by the distance (measured in $\varSpace$) from the origin to each hyperplane defined by a proximal normal vector: 
$$ R(f) = \inf\{\|x'\| \mid 0\neq (\zeta,\delta)\in N^P_{\hypo f}(x,u),\  (\zeta,\delta)^T((x',0) - (x,u))=0\}.$$
The following proposition gives the exact Lipschitz constant in terms of $R(f)$.
\begin{proposition}\label{prop:Lipschitz}
	Consider any upper semicontinuous, strictly upper radial $f$ where all $y\in\varSpace$ have $\lim_{v\rightarrow 0} v\cdot f(y/v)=0$.
	Then $f^\Gamma$ is $1/R(f)$-Lipschitz continuous.
\end{proposition}
\begin{proof}
	The key observation here is that for any $(\zeta,\delta)\in N^P_{\hypo f}((x,u))$,
	\begin{align}
	(\zeta,\delta)^T(x,u)
	&= \inf\left\{\zeta^Tx'\mid (\zeta,\delta)^T((x',0) - (x,u))=0\right\}\nonumber\\
	&= \|\zeta\|\inf\left\{\|x'\| \mid (\zeta,\delta)^T((x',0) - (x,u))=0\right\} \nonumber\\
	&\geq \|\zeta\|R(f)\nonumber
	\end{align}
	where the first equality is trivial and the second uses that the minimum norm point in this hyperplane will be a multiple of $\zeta$. Then the subgradient formula~\eqref{eq:proximal-subgradient-formula} ensures any $\zeta'\in\partial_P f^\Gamma(y)$ must have
	$$ \|\zeta'\| = \frac{\|\zeta\|}{(\zeta,\delta)^T(x,u)}\leq 1/R(f)$$
	for $(x,u)=\Gamma(y,\radTransSup{f}(y))$ and some $(\zeta,\delta)\in N^P_{\hypo f}((x,u))$. Since every radially dual subgradient is uniformly bounded, $f^\Gamma$ is uniformly Lipschitz. Considering a sequence of $(\zeta,\delta)\in N^P_{\hypo f}((x,u))$ approaching attainment of $R(f)$ makes this argument tight. \qed
\end{proof}

The condition $(x,u)^T(\zeta,\delta) \geq R(f)\|\zeta\|$ can be viewed as a natural way to quantify how radial $f$ is by strengthening~\eqref{eq:radial-NS-condition}. 
When $f$ is concave, $R(f)$ can be simplified.{\color{blue}
	\begin{lemma}\label{lem:concave-R}
		For any upper semicontinuous, upper radial, concave $f$,
		\begin{equation}\label{eq:concave-Lipschitz-constant}
		R(f) = \inf\{\|x\| \mid f(x)=0\}.
		\end{equation}
	\end{lemma}
	\begin{proof}
		By the concavity of $f$, $(\zeta,\delta)^T((x',0) - (x,u))=0$ implies $x' \in \closure\{ \bar x \mid f(\bar x)=0\}$. Hence $R(f) \geq \inf\{\|\bar x\| \mid f(\bar x)=0\}$.
		Any $\bar x$ with $f(\bar x)=0$ has $(\bar x,0)$ separated from $\hypo f$ by a supporting hyperplane with normal $(\zeta,\delta)$ at some $(x,u)\in \hypo f$. Hence $(\zeta,\delta)^T((x',0) - (x,u))=0$ separates $\bar x$ from $0$. So $\|\bar x\|\geq R(f)$. \qed
	\end{proof}
}
This matches the Lipschitz constants used in the previous works~\cite{Renegar2016,Renegar2019,Grimmer2017-radial-subgradient}.
This gives a natural way to measure the extent of radiality of a concave function by strengthening~\eqref{eq:concave-strict-radial-NS-condition}. From this, we see any concave maximization problem (with a known point in the interior of its domain) can be translated and transformed into a convex minimization problem that is uniformly Lipschitz continuous with constant depending on how interior the known point is to the function's domain.

\subsection{Smoothness of the Radially Dual Problem}
We say a continuously differentiable function $f$ is uniformly $L$-smooth if its gradient is $L$-Lipschitz continuous: for all $x,x'\in\dom f$
$$ \|\nabla f(x) - \nabla f(x')\| \leq L \|x-x'\|.$$

As an example, consider the radial dual of the continuously differentiable function $f(x)=\sqrt{(1-x^TQx)_+}$, which is upper radial for any matrix $Q$. This radially transforms into the similarly shaped function
\begin{align*}
f^\Gamma(y) = \sup\{ v>0 \mid v^2-y^TQy \leq 1\} = \sqrt{(1+y^TQy)_+}.
\end{align*}
Supposing $Q$ is positive semidefinite and nonzero, our primal is concave and differentiable on its domain but fails to have a Lipschitz gradient since $\nabla f(x)$ blows up at the boundary of its domain. However, in this case, the radially dual $f^\Gamma$ is well behaved, being convex and $\lambda_{max}(Q)$-smooth. 

For generic functions, we cannot hope to find smoothness out of thin air (like we do in the above example or quite generically with Lipschitz continuity in the previous section). This is due to~\eqref{eq:differentiable-duality} which establishes differentiability is preserved under the radial transformation. 
In line with this equivalence, we find that when $f$ is $L$-smooth, $\radTransSup{f}$ is $O(L)$-smooth, provided the domain of $f$ is bounded.
Let $D(f) = \sup\{\|x\| \mid x\in\dom f\}$ denote the norm of the largest point in the domain of $f$. Note that since we are primarily taking the radial dual of maximization problems that are bounded above and truncated below to be nonnegative optimization, $D(f)$ can be viewed as bounding the level set $\dom f = \{x \mid f(x)>0\}$.

The following proposition shows the operator norm of the radial transformation's Hessian is controlled by the ratio between $D(f)$ and $R(f)$ and the norm of the primal Hessian. From this, we conclude for twice differentiable $L$-smooth functions, the radial dual is also $O(L)$-smooth.
\begin{proposition}\label{prop:hessian}
	Consider any upper radial $f$ with $D(f)<\infty$ and $R(f)>0$ and $x,y\in\varSpace$ satisfying $(x,f(x))=\Gamma(y,\radTransSup{f}(y))$. If $f$ is twice continuously differentiable around $x$, then
	$$\|\nabla^2 \radTransSup{f}(y)\| \leq \left(1+\frac{D(f)}{R(f)}\right)^3 \|\nabla^2 f(x)\|.$$
\end{proposition}
\begin{proof}
	First we verify that $(\nabla f(x), -1)^T(x,f(x))< 0$ holds for all $x\in\dom f$ and so the Hessian formula~\eqref{eq:hessian-formula} applies: if $\nabla f(x)=0$, $(\nabla f(x), -1)^T(x,f(x))= -f(x) <0$ and if $\nabla f(x)\neq 0$, $(\nabla f(x), -1)^T(x,f(x))\leq -\|\nabla f(x)\|R(f) <0$.
	Then our bound on the Hessian of $\radTransSup{f}$ follows from the following pair of inequalities. First, the matrix $J = I -\frac{\nabla f(x)x^T}{(\nabla f(x),-1)^T(x,f(x))}$ has operator norm bounded by
	\begin{align*}
	\|J\| &\leq 1 + \frac{\|\nabla f(x)\|\|x\|}{|(\nabla f(x),-1)^T(x,f(x))|}
	\leq 1+\frac{\|x\|}{R(f)}
	\end{align*}
	and second, the Hessian formula's coefficient is similarly bounded by
	\begin{align*}
	\frac{f(x)}{(\nabla f(x),-1)^T(x,f(x))} &= 1 - \frac{\nabla f(x)^Tx}{(\nabla f(x),-1)^T(x,f(x))}\\
	&\leq 1 + \frac{\|\nabla f(x)\|\|x\|}{|(\nabla f(x),-1)^T(x,f(x))|}\\
	&\leq 1+\frac{\|x\|}{R(f)}.	
	\end{align*}
	Bounding each term in the Hessian formula~\eqref{eq:hessian-formula} gives the claimed result. \qed
\end{proof}

\begin{corollary}\label{cor:smoothness}
	Consider any upper radial, twice continuously differentiable $f$ with $D(f)<\infty$ and $R(f)>0$. If $f$ is $L$-smooth, $f^\Gamma$ is $\left(1+\frac{D(f)}{R(f)}\right)^3L$-smooth.
\end{corollary}
\begin{proof}
	For a twice continuously differentiable function, having $L$-Lipschitz gradient is equivalent to having Hessian bounded in operator norm by $L$. Noting that $R(f)>0$ implies $f$ is strictly upper radial by~\eqref{eq:differentiable-strict-radial-NS-condition}, we have a bijection between the domains of $f$ and $f^\Gamma$ from~\eqref{eq:domain-bijection}. Hence the Hessian of $f^\Gamma$ is uniformly bounded by $\left(1+\frac{D(f)}{R(f)}\right)^3L$. \qed
\end{proof}

Although this result requires smoothness of the primal objective $f$ to be maximized, it still provides an algorithmically valuable tool due to the symmetry-breaking nature of considering functions on the extended positive reals $\extPos$. Supposing $f$ is bounded above, this result allows us to extend the smoothness of $f$ on a level set $\dom f = \{x \mid f(x) >0\}$ to global smoothness of the dual $f^\Gamma$ on $\dom f^\Gamma = \varSpace$.

For example, consider an unconstrained $S=\RR^n$ instance of our previous motivating example of the Poisson likelihood problem~\eqref{eq:poisson-objective} which is not defined everywhere (only on $\{x \mid a_i^Tx >0\}$) with gradients blowing up as $x$ approaches the boundary of this domain. However, provided the measurements $\{a_i\}$ span $\RR^n$, this objective has bounded level sets. Consequently, for any twice continuously differentiable $r(x)$, our radial duality provides a reformulation that extends the smoothness on the level set $\{x \mid \mathcal{L}(x)-r(x)>0\}$ to hold globally. 


\subsection{Growth Conditions in the Radially Dual Problem}

For a lower semicontinuous function $f\colon \varSpace \rightarrow \extPos$, we say the {\L}ojasiewicz condition holds at a local minimum $x^*$ if for some constants $r>0$, $C>0$ and exponent $\theta\in [0,1)$, all nearby $x\in B(x^*,r)$ have
\begin{equation} \label{eq:Lojasiewicz-sub}
\dist(0,\partial_P f(x)) \geq C(f(x)-f(x^*))^{\theta}.
\end{equation}
For an upper semicontinuous function $f$ with local maximum $x^*$, we instead require all nearby $x\in B(x^*,r)$ have
\begin{equation} \label{eq:Lojasiewicz-sup}
\dist(0,\partial^P f(x)) \geq C(f(x^*)-f(x))^{\theta}.
\end{equation}
These conditions are widespread, holding for generic subanalytic functions~\cite{Lojasiewicz1963,Lojasiewicz1993} and nonsmooth subanalytic convex functions~\cite{Bolte2007}. These properties are closely related to the Kurdyka-{\L}ojasiewicz (KL) condition~\cite{Kurdyka1998} and H\"olderian growth/error bounds used by~\cite{Bolte2017,Yang2018,Roulet2020,RenegarGrimmer2018}, which are known to speed up the convergence of many first-order methods. 

Under mild conditions, the {\L}ojasiewicz condition is preserved by our radial transformation. Consequently, optimization algorithms based on solving the radially dual problem can enjoy the same improved convergence historically expected in the primal from such conditions. 

\begin{proposition}\label{prop:lojasiewicz}
	Consider any upper semicontinuous, strictly upper radial function $f$ with $R(f)>0$, $\sup f\in\RR_{++}$ {\color{blue} and points $(y^*,f^\Gamma(y^*))=\Gamma(x^*,f(x^*))$}. If $f$ satisfies the {\L}ojasiewicz condition~\eqref{eq:Lojasiewicz-sup} at $x^*$ with exponent $\theta$, then $f^\Gamma$ at $y^*$ satisfies the {\L}ojasiewicz condition~\eqref{eq:Lojasiewicz-sub} with the same exponent $\theta$.
\end{proposition}
\begin{proof}
	Let $r, C, \theta$ satisfy the {\L}ojasiewicz condition of $f$ at $x^*$ and denote the radially dual point as $y^*=x^*/f(x^*)$. Since $f$ is bounded above, $f^\Gamma$ is $1/R(f)$-Lipschitz continuous by Proposition~\ref{prop:Lipschitz}. Then every $0<r'<f^\Gamma(y^*)R(f)$ has $y\in B(y^*,r')$ map to $x=y/f^\Gamma(y)$ with
	\begin{align*}
	\|x-x^*\| &= \left\|\frac{y}{f^\Gamma(y)}-\frac{y^*}{f^\Gamma(y^*)}\right\|\\
	&\leq \left\|\frac{y-y*}{f^\Gamma(y)}\right\| +\left\|\frac{y^*}{f^\Gamma(y)}-\frac{y^*}{f^\Gamma(y^*)}\right\|\\
	& = \frac{\|y-y^*\|}{f^\Gamma(y)} + \frac{\|y^*\||f^\Gamma(y)-f^\Gamma(y^*)|}{f^\Gamma(y)f^\Gamma(y^*)}\\
	& \leq \frac{r'}{f^\Gamma(y)} + \frac{\|y^*\|r'}{R(f)f^\Gamma(y)f^\Gamma(y^*)}\\
	& \leq \frac{r'}{f^\Gamma(y^*)-r'/R(f)} + \frac{\|y^*\|r'}{R(f)(f^\Gamma(y^*)-r'/R(f))f^\Gamma(y^*)}.
	\end{align*}
	Therefore selecting small enough $r'$ guarantees that all of the dual points near $y^*$ map back to primal points $x=y/f^\Gamma(y)$ in the ball $B(x^*,r)$ where the {\L}ojasiewicz condition holds. Further, the Lipschitz continuity of the radial dual allows us to guarantee that all of these primal points have $f(x)$ bounded below by nearly $f(x^*)$ as	
	$$f(x)=f^{\Gamma\Gamma}(x)\geq 1/f^\Gamma(y) \geq 1/(f^\Gamma(y^*)-r'/R(f)) = (f(x^*)^{-1} + (R(f)/r')^{-1})^{-1}. $$
	Combining this with the assumed upper semicontinuity of $f$, we have $f(x)\rightarrow f(x^*)$ as $y\rightarrow y^*$ (despite not assuming continuity of the primal function $f$).
	
	Then all that remains is to show the {\L}ojasiewicz supgradient norm lower bound from the primal extends to lower bound the norm of the radially dual subgradients. For every $y\in B(y^*,r')$, the formula~\eqref{eq:proximal-subgradient-formula} ensures every $\zeta'\in\partial_P f^\Gamma(y)$ has
	$ \zeta'=\zeta/(\zeta,\delta)^T(x,u) $
	where $(x,u)=\Gamma(y,f^\Gamma(y))$ and $(\zeta,\delta)\in N^P_{\hypo f}((x,u))$. First, suppose $\delta\neq 0$. Then $u=f(x)$ and $-\zeta/\delta\in\partial^Pf(x)$ is a primal supgradient. Consequently, radially dual subgradients have size of at least
	\begin{align*}
	\|\zeta'\| &= \frac{\|\zeta/\delta\|}{(\zeta/\delta,1)^T(x,f(x))}\\
	&\geq \frac{\|\zeta/\delta\|}{\|\zeta/\delta\|\|x\|+f(x)}\\
	&\geq \frac{C(f(x^*)-f(x))^\theta}{C(f(x^*)-f(x))^\theta\|x\| + f(x)}\\
	&\geq \frac{Cf^\theta(x)f^\theta(x^*)}{C(f(x^*)-f(x))^\theta\|x\| + f(x)}\left(f^\Gamma(y)-f^\Gamma(y^*)\right)^\theta
	\end{align*}
	where the final inequality uses that $f(x)\geq 1/f^\Gamma(y)$ and $f(x^*)=1/f^\Gamma(y^*)$. Recalling that as $y\rightarrow y^*$, the related primal point $x=y/f^\Gamma(y)\rightarrow x^*$ and $f(x)\rightarrow f(x^*)$, the coefficient above must converge to a positive constant
	$$ \frac{Cf^\theta(x)f^\theta(x^*)}{C(f(x^*)-f(x))^\theta\|x\| + f(x)} \rightarrow \frac{Cf^{2\theta}(x^*)}{C0^\theta\|x^*\| + f(x^*)}. $$
	
	The boundary case of horizontal normal vectors with $\delta=0$ follows from the same argument above by passing to a sequence of points $(x_i,f(x_i))\rightarrow (x,f(x))$ and proximal normal vectors $(\zeta_i,\delta_i)\in N^P_{\hypo f}((x_i,f(x_i)))$ with $(\zeta_i,\delta_i)\rightarrow (\zeta,\delta)$ and $\delta_i\neq 0$. The existence of such a sequence is guaranteed by the Horizontal Approximation Theorem~\cite[Page 67]{Clarke1998-nonsmoothanalysis}. \qed
\end{proof}

The case of $\theta=0$ above is the important special case of sharpness. If this condition holds globally,~\eqref{eq:Lojasiewicz-sub} and~\eqref{eq:Lojasiewicz-sup} correspond to the global error bounds
\begin{equation}\label{eq:sharp-lower-bound}
f(x) \geq f(x^*) + C\|x-x^*\|	
\end{equation}
and
\begin{equation}\label{eq:sharp-upper-bound}
f(x) \leq \left(f(x^*) - C\|x-x^*\|\right)_+
\end{equation}
respectively. This condition has a long history in nonsmooth optimization (see Burke and Ferris~\cite{Burke1993} as a classic reference establishing the prevalence of sharp minima).
The two global sharp error bounds~\eqref{eq:sharp-upper-bound} and~\eqref{eq:sharp-lower-bound} are dually related.
\begin{proposition} \label{prop:sharpness}
	For any upper semicontinuous, strictly upper radial $f$ {\color{blue} with points $(y^*,f^\Gamma(y^*))=\Gamma(x^*,f(x^*))$} satisfying~\eqref{eq:sharp-upper-bound} at $x^*\in\varSpace$ with constant $C$, then $f^\Gamma$ satisfies~\eqref{eq:sharp-lower-bound} at $y^*$ with constant
	$ C/(C\|x^*\|+f(x^*)).$
\end{proposition}
\begin{proof}
	Denote the assumed upper bound on $f$ from sharpness as
	$ h(x) := f(x^*) - C\|x-x^*\|$.
	Then $h_+$ must be strictly  upper radial due to~\eqref{eq:concave-strict-radial-NS-condition} since $h$ is concave with $h(0)= 2h(x^*/2) - h(x^*)\geq 2f(x^*/2) - f(x^*)> 0$
	where the first equality uses that $h$ is linear on the segment $[0,x^*]$, the inequality uses that $h(x^*/2)\geq f(x^*/2)$, and the strict inequality uses that $f$ is strictly upper radial.
	The upper radial transformation $h^\Gamma_+$ is lower bounded by our claimed sharpness lower bound for any $y\in\varSpace$
	$$ h^\Gamma_+(y) \geq \frac{1}{f(x^*)} + \frac{C\left\|y - x^*/f(x^*)\right\|}{f(x^*)+C\|x^*\|} = f^\Gamma(y^*) + \frac{C\left\|y - y^*\right\|}{f(x^*)+C\|x^*\|} $$
	since $h^p_+(y,v)$ at $v=\frac{1}{f(x^*)} + \frac{C\left\|y - x^*/f(x^*)\right\|}{f(x^*)+C\|x^*\|}$ is at most
	\begin{align*}
	&\left(\frac{1}{f(x^*)} + \frac{C\left\|y - y^*\right\|}{f(x^*)+C\|x^*\|}\right)f(x^*) - C\left\|y-\left(\frac{1}{f(x^*)} + \frac{C\left\|y - y^*\right\|}{f(x^*)+C\|x^*\|}\right)x^*\right\|\\
	&\leq 1 + \frac{C\left\|y - y^*\right\|}{f(x^*)+C\|x^*\|}f(x^*) - C\left\|y-y^*\right\| + \frac{C^2\left\|y - y^*\right\|\|x^*\|}{f(x^*)+C\|x^*\|}\\	
	&=1 + C\left\|y - y^*\right\|\left(\frac{f(x^*)}{f(x^*)+C\|x^*\|} - 1 + \frac{C\|x^*\|}{f(x^*)+C\|x^*\|}\right)
	= 1	
	\end{align*}
	where $y^*=x^*/f(x^*)$ and the inequality uses the reverse triangle inequality.
	Using \cite[Lemma 4]{Grimmer2021-part1}, $f \leq h_+$ implies $f^\Gamma \geq h^\Gamma_+$, completing our proof. \qed
\end{proof}

	\section{Radial Algorithms for Concave Maximization}\label{sec:convex}
Now we turn our attention to understanding the primal convergence guarantees that follow from algorithms minimizing the radial dual. In this section, we consider concave maximization problems where being strictly upper radial and having $R(f)>0$ hold without loss of generality via a simple translation. 

We first remark on the natural measure of optimality in the primal that arises from considering the radial dual. Recall the set of fixed points of $\Gamma$ are exactly the horizontal line at height one $\{(y,1) \mid x\in \varSpace \}=\Gamma\{(x,1)\mid x\in\varSpace\}$. Consequently, a natural way to relate nearly optimal solutions between the primal and radial dual comes from considering when $\sup f =\inf f^\Gamma=1$. In this case, finding a dual point with accuracy
$ f^\Gamma(y_k)- \inf f^\Gamma \leq \epsilon$
{\color{blue} implies} a relative accuracy primal guarantee of
$$ \frac{\sup f-f(x_k)}{f(x_k)} \leq \epsilon. $$
using that $1/f^\Gamma(y_k)\leq f^{\Gamma\Gamma}(x_k)=f(x_k)$ for $x_k=y_k/f^\Gamma(y_k)$ on any upper radial $f$. Following this, we state all of our radial algorithm convergence guarantees in relative terms.

Secondly, we remark on the meaning of finding a radially dual solution minimized to zero objective value $f^\Gamma(y)=0$. In this case, $y$ certifies that the primal maximization is unbounded as the ray $(y,1)/v\in\epi f$ for all $v>0$. Note the converse of this is not true: for example, the strictly radial function $f(x)=\sqrt{(x+1)_+}$ is unbounded above, but has $f^\Gamma(y)>0$ everywhere.

\subsection{Radial Subgradient Method}\label{subsec:subgradient}
We begin by considering the radial subgradient method previously defined in Algorithm~\ref{alg:radial-subgradient-method}. This method simply takes the radial dual, applies the classic subgradient method to the resulting minimization problem, and then takes the radial dual again to return a primal solution.
Importantly this method is projection-free since any primal constraint set $S$ appears in the radial dual objective through its gauge $\gamma_S$.
This method is very similar to those considered in~\cite{Renegar2016,Grimmer2017-radial-subgradient} which also apply a subgradient method to a radial reformulation. However, those methods include additional steps periodically rescaling their radial objective. Our algorithm omits such steps while matching the improved convergence guarantees of~\cite{Grimmer2017-radial-subgradient}.

The standard subgradient method analysis shows the radial subgradient iterates $y_k$ converge in terms of radial dual optimality at a rate controlled by the radially dual Lipschitz constant. Recall that translating a point in the interior of $\hypo f$ to the origin ensures $R(f)>0$ {\color{blue} (by Lemma~\ref{lem:concave-R})} and so the radial dual is Lipschitz continuous (by Proposition~\ref{prop:Lipschitz}). Consequently, no structure needs to be assumed beyond concavity to analyze the radial subgradient method.

\begin{theorem}\label{thm:convex-radial-subgradient}
	Consider any upper semicontinuous, concave $f$ with $R(f)>0$ and $p^* = \sup f\in\RR_{++}$ attained on some nonempty set $X^*\subseteq\varSpace$. Then the radial subgradient method (Algorithm~\ref{alg:radial-subgradient-method}) with stepsizes $\alpha_k$ has primal solutions $x_k=y_k/f^\Gamma(y_k)$ satisfy
	$$\min_{k<T}\left\{\frac{p^*-f(x_k)}{f(x_k)}\right\} \leq \frac{\dist(p^*y_0,X^*)^2+\sum_{k=0}^{T-1}(p^*\alpha_k/R(f))^2}{2\sum_{k=0}^{T-1}p^*\alpha_k}.$$
	Selecting $x_0=0$ and $\alpha_k=\epsilon f^\Gamma(y_k)/\|\zeta'_k\|^2$ for any $\epsilon>0$ ensures
	$$ T \geq \frac{\dist(x_0,X^*)^2}{R(f)^2\epsilon^2} \quad \implies \quad \frac{1}{T} \sum_{k=0}^{T-1}\frac{p^*-f(x_k)}{p^*} \leq \epsilon.$$
\end{theorem}
\begin{proof}
	Having $R(f)>0$ ensures $f$ is strictly upper radial by~\eqref{eq:concave-strict-radial-NS-condition}. Then $f^\Gamma$ is convex by~\eqref{eq:convexity-duality} and has minimum value $d^*=1/p^*$ attained on $Y^*:=X^*/p^*$ by~\eqref{eq:minimizer-conversion}.
	The classic convex convergence analysis of subgradient methods follows from the fact that: for any $y^*\in Y^*$,
	\begin{align*}
	\|y_{k+1}-y^*\|^2 & = \|y_{k}-y^*\|^2 -2\alpha_k\zeta_k'^T(y_k-y^*) + \alpha_k^2\|\zeta_k'\|^2\\
	&\leq \|y_{k}-y^*\|^2 -2\alpha_k(f^\Gamma(y_k)-d^*) + \alpha_k^2\|\zeta_k'\|^2
	\end{align*}
	and so inductively,
	\begin{equation}\label{eq:subgradient-base}
	\sum_{k=0}^{T-1} \alpha_k(f^\Gamma(y_k)-d^*) \leq \frac{\|y_0-y^*\|^2+\sum_{k=0}^{T-1}\alpha_k^2\|\zeta_k'\|^2}{2}.
	\end{equation}
	Noting $(x_k,u_k)=\Gamma(y_k,f^\Gamma(y_k))$, the primal iterates have $f(x_k)\geq 1/f^{\Gamma}(y_k)$. Then multiplying through by $(1/d^*)^2$, which equals $(p^*)^2$, yields
	\begin{align*}
	\sum_{k=0}^{T-1}\frac{\alpha_k}{d^*}\left(\frac{p^*-f(x_k)}{f(x_k)}\right) &= \sum_{k=0}^{T-1}\frac{\alpha_k}{(d^*)^2}\left(\frac{1}{f(x_k)}-\frac{1}{p^*}\right)\\
	&\leq \frac{\|y_0/d^*-y^*/d^*\|^2+\sum_{k=0}^{T-1}(\alpha_k/d^*)^2\|\zeta_k'\|^2}{2}.
	\end{align*}
	Since $f^\Gamma$ is $1/R(f)$-Lipschitz (by Proposition~\ref{prop:Lipschitz}), every radially dual subgradient is uniformly bounded by $\|\zeta_k'\|\leq 1/R(f)$. Then selecting $y^*=\proj_{Y^*}(y_0)$ gives our claimed primal convergence rate.
	Observe that setting $x_0=0$ sets $y_0=x_0/f(x_0)=0$ as well. Then plugging $\alpha_k=\epsilon f^\Gamma(y_k)/\|\zeta'_k\|^2$ into~\eqref{eq:subgradient-base} yields
	\begin{align*}
	\frac{\dist(x_0,X^*)^2}{2} = \frac{\dist(y_0/d^*,X^*)^2}{2}
	&\geq \sum_{k=0}^{T-1}\frac{\alpha_k}{d^*}\left(\frac{f^\Gamma(y_k)-d^*}{d^*} - \frac{1}{2}\left(\frac{\alpha_k}{d^*}\right)\|\zeta_k'\|^2\right)\\
	&\geq \sum_{k=0}^{T-1}\epsilon\left(\frac{ f^\Gamma(y_k)}{d^*\|\zeta_k'\|}\right)^2\left(\frac{p^*-f(x_k)}{p^*} - \frac{\epsilon}{2}\right)\\
	&\geq \sum_{k=0}^{T-1}\epsilon R(f)^2\left(\frac{p^*-f(x_k)}{p^*} - \frac{\epsilon}{2}\right)
	\end{align*}
	{\color{blue} where the last line bounds $1/\|\zeta_k'\|^2$ below by $R(f)^2$ and $f^\Gamma(y_k)/d^*$ below by one.} Rearranging this completes our proof. \qed
\end{proof}
Recall for concave $f$ the formula for $R(f)$ can be simplified to $\inf\{\|x\|\mid f(x)=0\}$, which quantifies how interior the origin is to the set $\{x \mid f(x)>0\}$. In this light, the constants in this rate agree with those in the guarantees of~\cite{Grimmer2017-radial-subgradient}, up to small constants. 

The classic convergence rates of the subgradient method improve in the presence of growth conditions like~\eqref{eq:Lojasiewicz-sub} or~\eqref{eq:sharp-lower-bound}. For example growth with exponent $\theta=1/2$ corresponds to the case of quadratic growth (generalizing strong convexity) and leads to faster $O(1/\epsilon)$ convergence, see~\cite{Bach2012} as a simple example. When $\theta=0$, sharp growth enables the classic subgradient method to converge linearly, as shown by Polyak~\cite{Polyak1969,Polyak1979} more than 50 years ago. 
Recalling that these quantities are preserved from primal to radial dual (Propositions~\ref{prop:lojasiewicz} and~\ref{prop:sharpness}), we find the same improvements to hold for our radial subgradient method. The following two theorems establish this speed up when $\theta=0$ and $\theta>0$, using the radially dual Polyak stepsize $\alpha_k=(f^\Gamma(y_k)-d^*)/\|\zeta_k'\|^2$.
\begin{theorem}\label{thm:sharp-convex-radial-subgradient}
	Consider any upper semicontinuous, concave $f$ with $R(f)>0$ and $p^* = \sup f\in\RR_{++}$ attained at $x^*\in\varSpace$. Fixing $\alpha_k=(f^\Gamma(y_k)-d^*)/\|\zeta_k'\|^2$, if $f$ satisfies the sharp growth condition~\eqref{eq:sharp-upper-bound}, then the radial subgradient method (Algorithm~\ref{alg:radial-subgradient-method}) has $x_k=y_k/f^\Gamma(y_k)$ satisfy
	$$ T \geq 4\left(\frac{p^*+C\|x^*\|}{CR(f)}\right)^2\log_2\!\left(\frac{p^*-f(x_0)}{f(x_0)\epsilon}\right)\!\implies\!\min_{k<T}\!\left\{\frac{p^*-f(x_k)}{f(x_k)}\right\}\leq \epsilon. $$
\end{theorem}
\begin{proof} 
	Plugging the stepsize choice $\alpha_k=(f^\Gamma(y_k)-d^*)/\|\zeta_k'\|^2$ into~\eqref{eq:subgradient-base} implies
	\begin{equation} \label{eq:polyak-helper}
	\sum_{k=0}^{T-1} \frac{(f^\Gamma(y_k)-d^*)^2}{2} \leq \frac{\|y_0-y^*\|^2}{2R(f)^2}
	\end{equation}
	where $y^*=x^*/p^*$ and Proposition~\ref{prop:Lipschitz} is used to bound $\|\zeta'_k\|\leq 1/R(f)$. Then the radially dual sharpness bound from Proposition~\ref{prop:sharpness} guarantees $\|y_0-y^*\|\leq \frac{p^*+C\|x^*\|}{C}(f^\Gamma(y_0)-d^*)$. Hence
	$$
	\frac{1}{T}\sum_{k=0}^{T-1} (f^\Gamma(y_k)-d^*)^2 \leq \frac{(p^*+C\|x^*\|)^2(f^\Gamma(y_0)-d^*)^2}{C^2R(f)^2T}.
	$$
	Therefore some $k\leq 4\left(\frac{p^*+C\|x^*\|}{CR(f)}\right)^2$ has halved the dual objective gap, $f^\Gamma(y_k)-d^* \leq (f^\Gamma(y_0)-d^*)/2$. Repeatedly applying this, we conclude that 
	$$ T \geq 4\left(\frac{p^*+C\|x^*\|}{CR(f)}\right)^2\log_2\left(\frac{f^\Gamma(y_0)-d^*}{\epsilon'}\right)$$
	implies $\min_{k<T}\left\{f^\Gamma(y_k)-d^*\right\}\leq \epsilon'$ for any $\epsilon'>0$.
	Considering $\epsilon'= \epsilon/p^*$ gives the claimed linear convergence rate. \qed
\end{proof}
This generalizes the linear convergence results of~\cite{Renegar2016} for linear programming. To the best of our knowledge, this is the first first-order method linear convergence guarantee for generic non-Lipschitz, sharp convex optimization.
\begin{theorem}\label{thm:growth-convex-radial-subgradient}
	Consider any upper semicontinuous, concave $f$ with $R(f)>0$ and $p^* = \sup f\in\RR_{++}$ attained at $x^*\in\varSpace$. Fixing $\alpha_k=(f^\Gamma(y_k)-d^*)/\|\zeta_k'\|^2$, if $f$ satisfies the {\L}ojasiewicz condition~\eqref{eq:Lojasiewicz-sup} with exponent $\theta>0$, then the radial subgradient method (Algorithm~\ref{alg:radial-subgradient-method}) has $x_k=y_k/f^\Gamma(y_k)$ satisfy
	$$T \geq O\left(1/\epsilon^{2\theta}\right) \quad \implies \quad \min_{k<T}\left\{\frac{p^*-f(x_k)}{f(x_k)}\right\}\leq \epsilon. $$
\end{theorem}
\begin{proof}
	By Proposition~\ref{prop:lojasiewicz}, the {\L}ojasiewicz condition~\eqref{eq:Lojasiewicz-sub} holds at the dual minimizer $y^*=x^*/p^*$ for some constants $r', C'$ with the same exponent $\theta$. Integrating this condition (as done in~\cite[Theorem 5]{Bolte2017}) ensures every $y\in B(y^*,r')$ has the following local error bound 
	\begin{equation} \label{eq:local-error-bound}
	f^\Gamma(y) - d^* \geq \left(C'(1-\theta)\|y-y^*\|\right)^{1/(1-\theta)}.
	\end{equation}
	The subgradient method must have some $y_{k_0}$ in the ball $B(y^*,r')$ with
	$$ k_0 \leq \left(\frac{\|y_0-y^*\|}{\left(C'(1-\theta)r'\right)^{1/(1-\theta)}R(f)}\right)^2 $$
	since~\eqref{eq:polyak-helper} ensures the average iterate has objective gap squared at most $\left(C'(1-\theta)r'\right)^{2/(1-\theta)}$.	
	Notice that the Polyak stepsize ensures the distance from the iterates $y_k$ to $y^*$ is nonincreasing as
	\begin{align*}
	\|y_{k+1}-y^*\|^2 & = \|y_{k}-y^*\|^2 -2\alpha_k\zeta_k'^T(y_k-y^*) + \alpha_k^2\|\zeta_k'\|^2\\
	&\leq \|y_{k}-y^*\|^2 -2\alpha_k(f^\Gamma(y_k)-d^*) + \alpha_k^2\|\zeta_k'\|^2\\
	&\leq \|y_{k}-y^*\|^2 - \frac{(f^\Gamma(y_k)-d^*)^2}{\|\zeta_k'\|^2} \leq \|y_{k}-y^*\|^2.
	\end{align*}
	Hence all $k\geq k_0$ have $y_k\in B(y^*, r')$ as well. Then our claimed convergence rate follows by bounding the number of iterations required to ensure the objective gap halves $f^\Gamma(y_{k_0+k})-d^* \leq (f^\Gamma(y_{k_0})-d^*)/2$. Applying the local error bound~\eqref{eq:local-error-bound} to~\eqref{eq:polyak-helper} initialized at $y_{k_0}$ implies
	$$
	\frac{1}{T}\sum_{k=0}^{T-1} (f^\Gamma(y_{k_0+k})-d^*)^2 \leq \frac{(C'(1-\theta))^2(f^\Gamma(y_{k_0})-d^*)^{2(1-\theta)}}{R(f)^2T}.
	$$
	{\color{blue} Thus at some $k_1\leq k_0 + 4\left(\frac{C'(1-\theta)}{R(f)}\right)^2/(f^\Gamma(y_{k_0})-d^*)^{2\theta}$, the radially dual objective gap must have halved. Inductively, let $k_{i+1} \leq k_i + 4\left(\frac{C'(1-\theta)}{R(f)}\right)^2/(f^\Gamma(y_{k_i})-d^*)^{2\theta}$ denote an iteration with half the dual objective value of $k_i$. Let $k_{j+1}$ denote the first of these iterations with $f^\Gamma(y_{k_{j+1}})-d^*$ less than a target accuracy $\epsilon'>0$. Then $f^\Gamma(y_{k_{i}})-d^* \geq 2^{j-i}\epsilon'$ for all $i\leq j$. Inductively applying the definition of $k_i$ implies 
		$$ k_{j+1} -k_0 \leq \sum_{i=0}^j 4\left(\frac{C'(1-\theta)}{R(f)}\right)^2\frac{1}{(2^{j-i} \epsilon')^{2\theta}} \leq \frac{4}{ 1-2^{-2\theta}}\left(\frac{C'(1-\theta)}{R(f)}\right)^2\left(\frac{1}{\epsilon'}\right)^{2\theta}. $$
		Setting $\epsilon'=\epsilon/p^*$ ensures $\frac{p^*-f(x_{k_{j+1}})}{f(x_{k_{j+1}})} \leq \epsilon$.} \qed
\end{proof}

The previous pair of convergence theorems relied on using a Polyak stepsize, which requires the often impractical knowledge of $d^*$. This can be remedied by replacing the simple subgradient method in Algorithm~\ref{alg:radial-subgradient-method} with a more sophisticated stepping scheme like~\cite{Johnstone2017} or restarting scheme like~\cite{Yang2018,Roulet2020,RenegarGrimmer2018} which all attain similar convergence guarantees.

\subsection{Radial Smoothing Method}\label{subsec:smoothing}
Now we turn our attention to the radial smoothing method previously defined as Algorithm~\ref{alg:radial-smoothing-method} in the context of smoothing the radial dual of our quadratic program. {\color{blue} More generally, we consider maximizing a minimum of several smooth concave functions $f_j\colon \varSpace\rightarrow \mathbb{R}\cup\{-\infty\}$ with bounded level sets over polyhedral constraints $Ax\leq b$. Translating any known strictly feasible $x_0$ in the interior of the domain of $f_j$, we have $f_j(0)>0$ and $b>0$. Then we consider the equivalent nonnegative primal maximization problem
	\begin{equation} \label{eq:smoothing-primal}
	p^* = \begin{cases}
	\max_x & \min\{(f_j)_+(x) \mid j=1,\dots, m_1\}\\
	\mathrm{s.t.}& a_i^Tx\leq b_i \quad \text{ for } i=1,\dots,m_2
	\end{cases}
	\end{equation}
	which has $R((f_j)_+)\geq R>0$ and $D((f_j)_+)\leq D<\infty$ and each $b_i>0$.}

Further, since each $f_j$ has bounded level sets, $f_j$ is $L$-smooth on the level set $\{ x \mid f_j(x)>0\}$ for some $\sup\{\|\nabla^2 f_j(x)\|\mid f(x)>0\}\leq L<\infty$. This objective is strictly upper radial with radial dual
\begin{equation} \label{eq:smoothing-dual}
d^* = \min_{y\in\varSpace} \max\left\{f^\Gamma_j(x), (a_i/b_i)^Ty \mid j\in\{1,\dots, m_1\}, i\in\{1,\dots, m_2\}\right\}.
\end{equation}

Then we consider the smoothing of this objective for any $\eta>0$ given by
\begin{equation}\label{eq:general-smoothing}
g_\eta(y) = \eta \log\left(\sum_{j=1}^{m_1}\exp\left(\frac{(f_j)^\Gamma_+(y)}{\eta}\right)+\sum_{i=1}^{m_2}\exp\left(\frac{a_i^Ty}{b_i\eta}\right)\right).
\end{equation}
Our radial smoothing method (Algorithm~\ref{alg:radial-smoothing-method}) proceeds by minimizing this smoothing with Nesterov's accelerated method to produce a radially dual solution with accuracy $O(\eta)$. Nearly any other fast iterative method could be employed here instead, which could then avoid needing knowledge of problem constants.
Converting this radial dual guarantee back to the primal problem gives the following primal convergence theorem.
\begin{theorem}
	Consider any problem of the form~\eqref{eq:smoothing-primal}. Fixing $L_\eta=(1+D/R)^3L+\frac{\max\{1/R^2,\|a_i/b_i\|\}}{\eta}$ and $x_0=0$, the radial smoothing method (Algorithm~\ref{alg:radial-smoothing-method}) has $x_k=y_k/ \max\{(f_j)^\Gamma_+(y_k), (a_i/b_i)^Ty_k\}$ feasible with
	\begin{align*}
	\frac{p^*-\min\{f_j(x_k)\}}{\min\{f_j(x_k)\}} \leq \frac{2L_\eta(1+\eta p^*\log(m_1+m_2))^2D^2}{p^*(k+1)^2} + \eta p^*\log(m_1+m_2).
	\end{align*}
	Setting $\eta=\epsilon/2\log(m_1+m_2)$ ensures the following $O(1/\epsilon)$ convergence rate
	\begin{align*}
	&k+1 \geq 2(1+ p^*\epsilon/2)D\sqrt{\frac{(1+D/R)^3L}{p^*\epsilon}+\frac{2\max\{1/R^2,\|a_i/b_i\|^2\}\log(m_1+m_2)}{p^*\epsilon^2}}\\
	& \implies \quad \frac{p^*-\min\{f_j(x_k)\}}{\min\{f_j(x_k)\}} \leq p^*\epsilon.
	\end{align*}
\end{theorem}
\begin{proof}
	Observe that all of the $m_1+m_2$ functions defining $g_\eta$ are convex (by~\eqref{eq:convexity-duality}), $\max\{1/R, \|a_i/b_i\|\}$-Lipschitz continuous (by Proposition~\ref{prop:Lipschitz}) and $(1+D/R)^3L$-smooth (by Corollary~\ref{cor:smoothness}).
	Then~\cite[Proposition 4.1]{Beck2012} ensures $g_\eta$ is convex, is $(1+D/R)^3L+\frac{\max\{1/R^2,\|a_i/b_i\|\}}{\eta}$-smooth, and closely follows the radially dual objective with every $y\in\varSpace$ satisfying
	\begin{equation}\label{eq:smoothing-gap}
	0\leq g_\eta(y) - \max\left\{(f_j)^\Gamma_+(y), (a_i/b_i)^Ty \right\} \leq \eta\log(m_1+m_2).
	\end{equation}
	Note that for any $s>0$, the related primal super-level set is bounded by
	$$\sup\{\|x\| \mid f_j(x)\geq s,\ a_i^Tx\leq b_i \} \leq D. $$
	Recalling $\epi f^\Gamma = \Gamma(\hypo f)$ from~\eqref{eq:upper-graph-bijection} bounds every dual sub-level set by
	$$\sup\{\|y\| \mid f^\Gamma_j(y)\leq 1/s,\ (a_i/b_i)^Ty\leq 1/s \} \leq D/s. $$
	In particular, considering $s=p^*=1/d^*$ shows every radial dual minimizer has norm bounded by $d^*D$. Then the upper bound from~\eqref{eq:smoothing-gap} ensures the $d^*+\eta\log(m_1+m_2)$ sub-level set of $g_\eta$ is nonempty and the lower bound from~\eqref{eq:smoothing-gap} allows us to bound this level set by
	$$\sup\{\|y\| \mid g_\eta(y) \leq d^*+\eta\log(m_1+m_2)\} \leq (d^*+\eta\log(m_1+m_2))D$$
	Therefore the distance from $y_0=0$ to a minimizer of $g_\eta$ is at most $(d^*+\eta\log(m_1+m_2))D$.
	
	Since $g_\eta$ is smooth and has a minimizer, applying the standard accelerated method convergence guarantee~\cite{Nesterov1983} guarantees the iterates of our radial smoothing method have
	$$ g_\eta(y_k) -\inf g_\eta \leq \frac{2L_\eta(d^*+\eta\log(m_1+m_2))^2D^2}{(k+1)^2}. $$
	Converting this guarantee in terms of our radially dual objective,~\eqref{eq:smoothing-gap} ensures
	\begin{align*}
	&\max\left\{(f_j)^\Gamma_+(y_k), (a_i/b_i)^Ty_k\right\} - d^* \\
	&\qquad \leq \frac{2L_\eta(d^*+\eta\log(m_1+m_2))^2D^2}{(k+1)^2} + \eta \log(m_1+m_2).
	\end{align*}
	Stating this to be in terms of $x_k=y_k/ \max\{(f_j)^\Gamma_+(y_k), (a_i/b_i)^Ty_k\}$ yields
	\begin{align*}
	&\frac{p^*-\min\{f_j(x_k)\}}{\min\{f_j(x_k)\}}\leq \frac{2L_\eta(1+\eta p^*\log(m_1+m_2))^2D^2}{p^*(k+1)^2} + \eta p^*\log(m_1+m_2). \tag*{\qed}
	\end{align*}
\end{proof}
Renegar~\cite{Renegar2019} uses the same general technique to give accelerated convergence guarantees for solving the broad family of hyperbolic programming problems (which includes semidefinite programming) where the radial dual also admits a natural smoothing.
The restarting schemes of~\cite{Roulet2020} and~\cite{RenegarGrimmer2018} both explicitly consider restarting smoothing methods to attain improved convergence when growth conditions like the {\L}ojasiewicz condition~\eqref{eq:Lojasiewicz-sub} hold. Due to Proposition~\ref{prop:lojasiewicz}, applying these more sophisticated methods to solve the radially dual problem will give rise to radial algorithms that enjoy the same improved convergence. The analysis of such a method should follow similarly to Theorem~\ref{thm:growth-convex-radial-subgradient}.

\subsection{Radial Accelerated Method}\label{subsec:accel}
Motivated by our example transforming the Poisson likelihood problem~\eqref{eq:poisson-objective}, algorithms can be designed to take advantage of the radial transformation extending smoothness on a level set to hold globally. Consider maximizing any twice differentiable concave function $f\colon \varSpace \rightarrow \RR\cup\{-\infty\}$ with bounded level sets. Then, without loss of generality, we have $0\in\interior \{x\mid f(x)>0\}$ and so $f_+$ is strictly upper radial. Letting $L=\sup\{\|\nabla^2 f(x)\|\mid f(x)>0\}$, Corollary~\ref{cor:smoothness} ensures $f^\Gamma_+$ is $(1+D(f)/R(f))^3L$-smooth on all of $\varSpace$. Hence $f^\Gamma_+$ can be minimized directly using Nesterov's accelerated method, giving the following {\it radial accelerated method} defined by Algorithm~\ref{alg:radial-accel-method}. This radial algorithm inherits the primal accelerated method's $O(\sqrt{L\dist(x_0,X^*)^2/\epsilon})$ rate, only requiring $L$-smoothness on the level set $\{x \mid f(x)>0\}$.
\begin{algorithm}[H] 
	\caption{The Radial Accelerated Method} 
	\label{alg:radial-accel-method} 
	\begin{algorithmic}[1] 
		\REQUIRE $f\colon \varSpace \rightarrow \extPos$, $x_0\in\dom f$, $L>0$, $T\geq 0$
		\STATE $(y_0, v_0) = \Gamma (x_0, f(x_0))$ and $\tilde y_0=y_0$ \hfill {\it Transform into the radial dual}
		\FOR{$k=0\dots T-1$}
		\STATE $\tilde y_{k+1} = y_k - \nabla f^\Gamma(y)/(1+D(f)/R(f))^3L$ \hfill {\it Run the accelerated method}
		\STATE $y_{k+1} = \tilde y_{k+1} + \frac{k-1}{k+2}(\tilde y_{k+1}-\tilde y_k)$ 
		\ENDFOR
		\STATE $(x_T, u_T) = \Gamma (y_T, \radTransSup{f}(y_T))$  \hfill {\it Transform back to the primal}
	\end{algorithmic}
\end{algorithm}

\begin{theorem}\label{thm:radial-accelerated}
	Consider any twice differentiable, concave $f$ with $R(f)>0$, $D(f)<\infty$, and $p^* = \sup f\in\RR_{++}$ attained on $X^*\subseteq\varSpace$. Fixing $x_0=0$, the radial accelerated method (Algorithm~\ref{alg:radial-accel-method}) has for any $\epsilon>0$,
	$$ k+1 \geq (1+D(f)/R(f))^{3/2}\sqrt{\frac{2L\dist(x_0,X^*)^2}{p^*\epsilon}} \quad \implies \quad \frac{p^*-f(x_k)}{f(x_k)}\leq \epsilon. $$
\end{theorem}
\begin{proof}
	Recall the $f^\Gamma$ is convex by~\eqref{eq:convexity-duality} and is $(1+D(f)/R(f))^3L$-smooth by Corollary~\ref{cor:smoothness}. Then Nesterov's classic analysis~\cite{Nesterov1983} ensures
	$$ f^\Gamma(y_k) -d^* \leq \frac{2(1+D(f)/R(f))^3L\dist(y_0, Y^*)^2}{(k+1)^2} $$	
	where $Y^*=X^*/p^*$. Letting $(x_k,u_k)=\Gamma(y_k,v_k)$ yields primal iterates with $ f(x_k)\geq 1/f^\Gamma(y_k)$. Then multiplying through by $1/d^*=p^*$ produces
	$$ \frac{p^*-f(x_k)}{f(x_k)}\leq \frac{2(1+D(f)/R(f))^3L\dist(y_0/d^*, X^*)^2}{p^*(k+1)^2}. $$	
	Noting that $y_0/d^*=x_0=0$, this gives the claimed convergence guarantee. \qed
\end{proof}
A few remarks on this convergence result. The additional coefficient of $(1+D(f)/R(f))^{3/2}$ is quite pessimistic as many of the examples we have considered have radial dual smoother than the primal, but Corollary~\ref{cor:smoothness} fails to capture this potential upside in its $O(L)$ bound. For particular applications, we expect much tighter bounds on the radially dual smoothness are possible.
The proposed radial accelerated method unrealistically relies on knowledge of our smoothness constant upper bound $(1+D(f)/R(f))^3 L$. However, this can be remedied by including a linesearch/backtracking as done in~\cite{Beck2009,Nesterov2015}.

Under growth conditions, the convergence of accelerated methods also improves. For example, applying the adaptive accelerated gradient method of~\cite{Liu2017} to solve the radially dual problem would give a radial method that speeds up in the presence of primal growth conditions by Proposition~\ref{prop:lojasiewicz}. The analysis of such a method should follow similarly to that of Theorem~\ref{thm:growth-convex-radial-subgradient}.
	\section{Radial Algorithms for Nonconcave Maximization}\label{sec:nonconvex}
Our radial duality theory applies beyond concave maximization problems, applying to the broader family of {\color{blue} nonconcave but upper radial maximization. Section~\ref{subsec:nc-obj} outlines several families of nonconvex settings where upper radiality holds and then Section~\ref{subsec:nc-algorithms} presents a performance guarantee for the radial subgradient method when maximizing a collection of such upper radial nonconcave functions.}  

\subsection{Examples of Radial Duality with Nonconvex Objectives or Constraints}\label{subsec:nc-obj}
{\color{blue} We say that a set $S\subseteq \varSpace$ is {\it star-convex with respect to the origin} if every $x\in S$ has the line segment $\lambda x\in S$ for all $0\leq \lambda\leq 1$. Geometrically, upper radial functions all have a star-convex{\color{blue}-like} hypograph with respect to the origin~\cite[Lemma 1]{Grimmer2021-part1}, meaning that all $(y,v)\in \hypo f$ have $\lambda(y,v)\in \hypo f$ for all $0 < \lambda\leq 1$\footnote{{\color{blue} Note this hypograph is not actually star convex since $(0,0)\not\in \hypo f$.}}.}
Star-convexity has been considered throughout the optimization literature. The structure of optimizing over star-convex constraint sets has been considered as early as~\cite{Rubinov1986}. In general, even linear optimization over star-convex bodies is NP-hard~\cite{Chandrasekaran2010}. Efficient global optimization of star-convex objectives is possible if star-convexity holds with respect to a global optimizer (see~\cite{Nesterov2006,Guminov2017,Lee2016,Guminov2019,Hinder2020}). \\

\noindent {\it Star-Convex Constraints.}
Star convexity w.r.t.~the origin is exactly the condition needed to ensure the {\color{blue} nonstandard} indicator $\hat\iota_{S}(x)= \begin{cases}
+\infty & \text{if\ } x\in S\\
0 & \text{if\ } x\not\in S.
\end{cases}$ is strictly upper radial\footnote{This is essentially by definition as $v\cdot \hat\iota_S(y/v)$ is nondecreasing in $v$ if and only if $S$ is star-convex w.r.t.~the origin. Then it is simple to check this function is upper semicontinuous and is vacuously strictly increasing on its effective domain $\dom\hat\iota_S = \emptyset$, which is empty.}. Then the radial dual of such a star-convex set's indicator function is given by the gauge
$$\hat\iota_S^\Gamma(y) = \sup\{v>0 \mid v\cdot \hat\iota_S(y/v)\leq 1\} = \sup\{v>0 \mid y/v\not\in S\}=\gamma_S(y).$$
Importantly, the gauge $\gamma_S(y)$ is convex if and only if $S$ is convex. As a result, algorithms utilizing the radial dual of star-convex constraints avoid needing difficult nonconvex orthogonal projections, replacing them with evaluating a nonconvex gauge function appearing in the objective.

One important example where star-convex sets arises comes from considering chance constraints~\cite{Haneveld2020,Nemirovski2007,Yuan2017}. Given some distribution over potential constraint sets $S_\xi\subseteq\varSpace$, a robust problem formulation may ensure that the constraint is satisfied with probability $\Lambda\in [0,1]$. Then the chance-constrained feasible region is $S=\{x \mid \mathbb{P}(x\in S_\xi) \geq \Lambda\}$. If each potential constraint set is convex with $0\in S_\xi$, then $S$ is star-convex w.r.t.~the origin.\\

\noindent {\it Optimization over Compact Sets.} Now we generalize our previous example from Section~\ref{sec:qp-example} where we saw that any nonconcave quadratic program with a compact polyhedral feasible region could be rescaled for our radial duality to apply. Consider maximizing any continuously differentiable function $f$ over a compact set $S$ that is star-convex w.r.t.~the origin. Supposing $f(0)>0$, this is equivalent to the following maximization problem of the primal form~\eqref{eq:base-problem}
$$ \max_{y\in \varSpace} \min\{ (1 + \lambda f(x))_+, \hat\iota_S(x)\} $$
for any $\lambda>0$. We check when this objective is strictly upper radial by considering whether its perspective function is strictly increasing on its domain:
\begin{align*}
v\cdot \min_i\left\{(1+\lambda f(y/v))_+,\ \hat\iota_{S}(y/v)\right\}
=\begin{cases}
\left(v+\lambda vf(y/v)\right)_+ & \text{ if } y/v\in S\\
0 & \text{ otherwise.}
\end{cases}
\end{align*}
The partial derivative of this with $v$ at any $y/v\in S\cap \dom (1+\lambda f)_+$ is
$$1 - \lambda (\nabla f(y/v),-1)^T(y/v, f(y/v)).$$
Noting that $(\nabla f(x),-1)^T(x, f(x))$ is a continuous function on the compact set $S\cap \dom (1+\lambda f)_+$, we can select $\lambda>0$ small enough to always have
$$1 - \lambda (\nabla f(y/v),-1)^T(y/v, f(y/v))>0.$$
Doing so makes our objective strictly upper radial and so radial duality applies.\\

\noindent {\it Nonconvex Regularization.} 
Many optimization tasks take the additive composite form
$ \max_{y\in\varSpace}  f(x) - r(x) $
where $f$ is an upper semicontinuous, concave function with $f(0)>0$ and $r(x)$ is an added (or rather subtracted since we are maximizing) regularization term. 
Many sparsity-inducing regularization penalties decompose as a sum over the $x$'s coordinates $r(x) = \sum_{i=1}^n \sigma(x_i)$ for some simple nonconvex function $\sigma:\RR\rightarrow\RR$. For example, $\ell_q$-regularization sets $\sigma(t)= \lambda|t|^q$ for some $0<q<1$, bridging the gap between $\ell_0$ and $\ell_1$-regularization. Many more regularizers are of this form, like SCAD regularization~\cite{Fan2001}, MCP~\cite{Zhang2010}, and firm thresholding~\cite{Gao1997}. See~\cite{Wen2018} for a wide survey. 

These regularizers are all continuous and have $r(y/v)$ nonincreasing in $v$. These two simple properties suffice to guarantee subtracting $r$ from $f$ will not break its upper radiality since
$$ v (f(y/v)-r(y/v))_+ = \max\{vf(y/v) - vr(y/v), 0\} $$
is a sum of two upper semicontinuous, nondecreasing functions in $v$. As a result, our radial duality applies to the nonconcave primal $(f(x)-r(x))_+$.\\

\noindent {\it Optimization with Outliers.}
Many learning problems take the form of minimizing a stochastic loss function $\EE_\xi \left[f(x,\xi)\right]$ using a finite sample approximation {\color{blue} with $f(\colon,\xi)\colon \varSpace\rightarrow \mathbb{R}\cup\{\infty\}$.} Given i.i.d.~samples $\xi_1,\dots,\xi_{s}$, this problem can be formulated as 
$$ \max_{x\in\varSpace}\frac{1}{s}\sum^s_{i=1} -f(x,\xi_i) . $$
If each $-f(\cdot, \xi_i)$ is concave and a point is known in the interior of each function's domain, a translation {\color{blue} can ensure $-f(0, \xi_i)>0$ for all $i$. Hence their sum is concave with a positive value at zero and so $\frac{1}{s}\left(\sum^s_{i=1} -f(x,\xi_i)\right)_+$ is upper radial.} Hence our radial duality can be applied. In the presence of $t$ outliers in the $s$ samples $\xi_1,\dots,\xi_{s}$, this finite sample approximation could be improved to only consider the loss function on the best $s-t$ samples
$$ \max_{x\in\varSpace} \max\left\{\frac{1}{s-t}\left(\sum_{i\in S} -f(x,\xi_i)\right)_+, \mid S\subseteq\{1...s\}, |S|=s-t\right\}. $$
{\color{blue} These partial sums are also concave with positive value at zero} and hence this whole objective is upper radial by~\cite[Corollary 2]{Grimmer2021-part1} and so our radial duality applies.
The minimax formulation of~\cite{Yu2014} exactly corresponds to this problem formulation at its equilibrium.
By the same corollary, our radial duality also applies to maximizing the $(s-t)$th largest element of $\{-f(x, \xi_i)\}^s_{i=1}$. Such an optimization problem captures the classic idea of least median of squares regression~\cite{Rousseeuw1984}.


\subsection{Example Nonconcave Guarantee for the Radial Subgradient Method}\label{subsec:nc-algorithms}
In this concluding section, we demonstrate the style of results possible from applying our radial duality to {\color{blue} upper radial }nonconcave maximization.
In particular, we consider the nonconcave, nonsmooth primal problem of maximizing the minimum of a set of twice continuously differentiable, strictly upper radial $f_j$ over some convex set $S\subseteq \varSpace$
\begin{equation} \label{eq:nc-primal}
p^* = \begin{cases}
\max_x & \min\{f_j(x) \mid j=1,\dots, m\}\\
\mathrm{s.t.}& x \in S
\end{cases} \quad = \max_{x\in\varSpace} \min\{f_j(x), \hat\iota_{S}(x)\}
\end{equation}
where each $f_j$ has $R(f_j)\geq R>0$ and bounded level sets $D(f_j)\leq D<\infty$ and the origin lies in the constraints with $B(0, R)\subseteq S$. Let $L\geq \sup\{\|\nabla^2 f_j(x)\| \mid f_j(x)>0, x\in S\}$ bound the smoothness of each $f_j$ on this compact level set.

This primal is strictly upper radial since each function defining the minimum is strictly upper radial. Then the radial dual of this problem is
\begin{equation}\label{eq:nc-dual}
d^* = \min_{y\in\varSpace} \max\{f_j^\Gamma(y), \gamma_S(y)\}.
\end{equation}
Note each $f^\Gamma_j(y)$ is convex if and only if $f_j$ is concave by~\eqref{eq:convexity-duality}. Hence if our primal~\eqref{eq:nc-primal} is nonconcave, our radial dual~\eqref{eq:nc-dual} will be nonconvex. Regardless, our previously proposed radial subgradient method (Algorithm~\ref{alg:radial-subgradient-method}) can still be applied and analyzed.

Recently, convergence theory for subgradient methods without convexity has been developed. Particularly, consider minimizing a nonconvex, nonsmooth function $g\colon \varSpace\rightarrow \RR$ that is bounded below. Then~\cite[Theorem 3.1]{Davis2018} ensures that provided $g$ is uniformly $M$-Lipschitz and $\rho$-weakly convex (defined as $g+\frac{\rho}{2}\|\cdot\|^2$ being convex), {\color{blue} the subgradient method $y_{k+1} = y_k - \alpha \zeta_k$ for $\zeta_k\in\partial_P g(y_k)$ has some $y_k$ that is nearly stationary on the Moreau envelope of $g$. In particular,~\cite[(3.9)]{Davis2018} this implies that proper selection\footnote{{\color{blue} Namely, given the method will be run for $T$ steps, the example analysis of~\cite{Davis2018} shows selecting $\alpha=\sqrt{\frac{g(y_0)-\inf g}{\rho M^2(T+1)}}$ suffices to give the claimed rate. Alternative stepsizes could be analyzed by the same proof technique proposed therein resulting in different assumptions on which parameters are known.}} of $\alpha$ ensures some $y_k$ has a nearby $y$ that is nearly stationary
	\begin{align}
	& T \geq \left\lceil\frac{16 \rho M^2(g(y_0)-\inf g)}{\epsilon^4}\right\rceil \nonumber\\
	& \implies \min_{k< T}\left\{\|y-y_k\|\right\}\leq 2\rho\epsilon \text{ with } \dist(0,\partial_P g(y))\leq \epsilon. \label{eq:nc-subgrad-rate}
	\end{align}
	
	Applying this to the radial dual allows us to ensure a nearly stationary point $y$ near a dual iterate $y_k$ exists. Then converting this guarantee back to the primal preserves the above $O(1/\epsilon^4)$ rate despite not assuming the primal~\eqref{eq:nc-primal} is either Lipschitz or weakly convex {\color{blue} (instead assuming it is strictly upper radial)}.
	\begin{theorem}
		Consider any problem of the form~\eqref{eq:nc-primal} with $p^*\in\RR_{++}$. Fixing $x_0=0$ and $\alpha_k=\epsilon/\|\zeta'_k\|^2 $, the radial subgradient method (Algorithm~\ref{alg:radial-subgradient-method}) with properly chosen constant stepsize $\alpha_k=\alpha$ has $x_k=y_k/ \max\{f^\Gamma_j(y_k), \gamma_S(y_k)\}$ satisfy
		\begin{align*}	
		&T\geq \left\lceil\frac{16(1+D/R)^3L(\min\{f_j(x_0)\}-p^*)}{R^2\min\{f_j(x_0)\}p^*\epsilon^4}\right\rceil\\
		&\implies \min_{k< T}\left\{\|x-x_k\|\right\}\leq 2p^*(1+D/R)^{4}L\epsilon\\
		&\qquad \ \text{ with \ } \dist(0,\partial^P \min\{f_j, \hat\iota_{S}\}(x))\leq \frac{p^*\epsilon}{1-\epsilon D}
		\end{align*}
		for some nearby $x\in\varSpace$ provided $0<\epsilon < 1/D$.
	\end{theorem}
	\begin{proof}
		Observe that each function in the maximum defining the radial dual~\eqref{eq:smoothing-dual} is $1/R$-Lipschitz (by Proposition~\ref{prop:Lipschitz}) and each $f^\Gamma_j$ is $(1+D/R)^3L$-smooth (by Corollary~\ref{cor:smoothness}). Then the whole radially dual objective $\max\{f^\Gamma_j(y), \gamma_S(y)\}$ is $1/R$-Lipschitz and $(1+D/R)^3L$-weakly convex.
		Hence even though our primal is not assumed to be either Lipschitz or weakly convex, these two properties occur in the radial dual due to each $f_i$ having $R(f_i)>0$ and smoothness on the level set $\{x\mid f_j(x)>0\}$ respectively.
		Then we can apply~\eqref{eq:nc-subgrad-rate} implying that whenever  $T\geq \left\lceil\frac{16(1+D/R)^3L(\min\{f^\Gamma_j(y_0)\}-d^*)}{R^2\epsilon^4}\right\rceil$, a nearby $y$ has
		\begin{align*}	
		&\min_{k< T}\left\{\|y-y_k\|\right\}\leq 2(1+D/R)^3L\epsilon\\
		&\text{ and } \dist(0,\partial_P \max\{f^\Gamma_j, \gamma_S\}(y))\leq \epsilon.
		\end{align*}
		
		First, we show the nearby radial dual solution $y$ maps to a primal solution $x=y/\max\{f^\Gamma_j(y),\gamma_S(y)\}$ that is near the primal iterates.
		Having dual distance $\|y-y_k\|\leq 2(1+D/R)^3L\epsilon$ ensures $\|x-x_k\|$ is bounded by
		\begin{align*}
		&\left\|\frac{y}{\max\{f^\Gamma_j(y),\gamma_S(y)\}}-\frac{y_k}{\max\{f^\Gamma_j(y_k),\gamma_S(y_k)\}}\right\|\\
		&\leq \frac{\|y-y_k\|}{\max\{f^\Gamma_j(y),\gamma_S(y)\}}+\left\|\frac{y_k}{\max\{f^\Gamma_j(y),\gamma_S(y)\}}-\frac{y_k}{\max\{f^\Gamma_j(y_k),\gamma_S(y_k)\}}\right\|\\
		&= \frac{\|y-y_k\|}{\max\{f^\Gamma_j(y),\gamma_S(y)\}} + \|x_k\|\left|\frac{\max\{f^\Gamma_j(y_k),\gamma_S(y_k)\}}{\max\{f^\Gamma_j(y),\gamma_S(y)\}}-1\right|\\
		&\leq \frac{\|y-y_k\|}{\max\{f^\Gamma_j(y),\gamma_S(y)\}} + \frac{D\|y-y_k\|/R}{\max\{f^\Gamma_j(y),\gamma_S(y)\}}\\
		&\leq p^*(1+D/R)\|y-y_k\|\leq 2p^*(1+D/R)^{4}L\epsilon
		\end{align*}
		where the first inequality uses the triangle inequality and the second uses the bounded primal level sets and the radially dual $1/R$-Lipschitz continuity, and the third uses that $d^*=1/p^*\in\RR_{++}$.
		
		We complete our proof by relating the stationarity of $y$ to that of $x$. Let $v=\max\{f^\Gamma_j(y),\gamma_S(y)\}$, $u=1/v$ and $\zeta'\in\partial_P\max\{f^\Gamma_j, \gamma_S\}(y)$ denote a radially dual subgradient with $\|\zeta'\|\leq \epsilon$. Then we can bound
		\begin{align*}
		(\zeta',-1)^T(y,v) \leq \|\zeta'\|\|y\| - v &\leq \epsilon\|x\|/u - 1/u \leq -(1-\epsilon D)/p^*<0
		\end{align*}
		using that $u\leq f(x_k)\leq p^*$. Note $\epi\max\{f^\Gamma_j,\gamma_S\} = \Gamma(\hypo\min\{f_j,\hat\iota_S\})$ by~\eqref{eq:upper-graph-bijection}. Then the normal $(\zeta',-1)\in N^P_{\epi\max\{f^\Gamma_j,\gamma_S\}}(y,v)$ corresponds to the primal normal
		$ (\zeta', (\zeta',-1)^T(y,v))\in N^P_{\hypo\min\{f_j, \hat\iota_{S}\}}(x,u) $ by~\cite[Proposition 5]{Grimmer2021-part1}.
		Hence $\zeta:=\zeta'/(\zeta',-1)^T(y,v)\in\partial^P\min\{f_j, \hat\iota_{S}\}(x)$ is a primal subgradient with norm at most $O(\epsilon)$ as
		\begin{equation*}
		\|\zeta\|=\left\|\frac{\zeta'}{(\zeta',-1)^T(y,v)}\right\| =\frac{\|\zeta'\|}{|(\zeta',-1)^T(y,v)|} \leq \frac{p^*\epsilon}{1-\epsilon D}. \tag*{\qed}
		\end{equation*}
	\end{proof}
}


	\begin{acknowledgements}
		The author thanks Jim Renegar broadly for inspiring this work and concretely for providing feedback an early draft and Rob Freund for constructive thoughts helping focus this work. Additionally, two anonymous referees and the associate editor provided useful feedback much improving this work's presentation and clarity.
	\end{acknowledgements}

	%
	%

	\bibliographystyle{spmpsci}      
	\bibliography{references}   

	\appendix
	{\color{blue}\section{LogSumExp Gradients and QP Optimality Certificates} \label{appendix}
	In our quadratic programming example~\eqref{eq:qp-smoothing} and our generalized setting~\eqref{eq:general-smoothing}, we consider smoothings of a finite maximum. Given a smooth convex functions $f_i\colon\varSpace \rightarrow \mathbb{R}$, we considered the smoothing of $\max\{f_i\}$ with parameter $\eta>0$ given by
	$ f_\eta(x) := \eta \log\left( \sum_{i=0}^n \exp(f_i(x)/\eta)\right).$
	Its gradient is given by
	$ \nabla f_\eta(x) = \sum \lambda_i \nabla f_i(x)$
	where $\lambda_i = \exp(f_i(x)/\eta) / \sum_j  \exp(f_j(x)/\eta)$. Computationally evaluating this requires mild care to avoid precision issues with exponentiating potentially larger numbers. It is numerically stable to instead compute these coefficients via the equivalent formula
	$$\lambda_i = \frac{\exp((f_i(x) - \max\{f_k(x)\})/\eta)}{\sum_j  \exp((f_j(x) - \max\{f_k(x)\})/\eta)} \ .$$
	
	Next, we specialize this formula to the setting of quadratic programming for $g_\eta$ in~\eqref{eq:qp-smoothing}. Observe the gradient of the objective component is given by
	$$ \nabla \left(\frac{c^Ty+1 + \sqrt{(c^Ty+1)^2 +2y^TQy}}{2}\right)_+\ (y)= \frac{Qx +c}{1-\frac{1}{2}x^TQx} $$
	where $x=y/\left(\frac{c^Ty+1 + \sqrt{(c^Ty+1)^2 +2y^TQy}}{2}\right)_+$  by using the gradient formula~\eqref{eq:gradient-formula}. The gradients of the transformed constraints are simply
	$ \nabla a_i^T y /b_i = a_i/b_i $.
	Then the gradient of the smoothing overall is given by
	$$\nabla g_\eta(y) = \lambda_0 \frac{Qx +c}{1-\frac{1}{2}x^TQx} + \sum_{i=1}^n \lambda_i a_i/b_i \ . $$
	This gradient can be computed using two matrix multiplications with $A$: $Ay$ is needed to compute the coefficients $\lambda_i$, then $A^T[\lambda_1/b_1 \dots \lambda_n/b_n]$ is needed for the summation above. This gradient formula indicates a reasonable selection of dual multipliers
	$ v_i = \frac{\lambda_i (1-\frac{1}{2}x^TQx)}{\lambda_0b_i}$
	as we then have $g_\eta(y)$ proportional to $Qx+c + A^Tv$.
}
\end{document}